\documentclass[11pt,leqno,oneside]{amsart}
\usepackage{amssymb,mathrsfs}
\usepackage{tikz}
\usepackage[normalem]{ulem}
\usepackage{upgreek}

\usepackage{color} 
      \definecolor{mydefi}{cmyk}{1,0,0,.5}
      \definecolor{myred}{rgb}{.7,.1,.1}
      \definecolor{myblue}{rgb}{.1,.1,.6}
      \definecolor{mygreen}{rgb}{.1,.6,.1}

\usepackage[urlbordercolor={1 1 1}, pdfborder={0 0 0}, bookmarks=true,
  colorlinks=true, linkcolor=myblue, citecolor=myblue,
  urlcolor=myblue, hyperfootnotes=false]{hyperref}

\usepackage[shortalphabetic,abbrev]{amsrefs} 

\newcommand{\N}{\mathbb{N}} 
\newcommand{\Z}{\mathbb{Z}} 
\newcommand{\C}{\mathbb{C}} 
\newcommand{\End}{\operatorname{End}} 
\newcommand{\Hom}{\operatorname{Hom}} 
\newcommand{\GL}{\mathrm{GL}} 
\newcommand{\B}{\mathfrak{B}} 
\newcommand{\Sym}{\mathfrak{S}} 
\newcommand{\sgn}{\operatorname{sgn}} 
\newcommand{\T}{\mathsf{T}} 
\newcommand{\U}{\mathsf{U}} 
\newcommand{\TS}{\mathsf{S}} 
\newcommand{\ov}{\overline} 
\newcommand{\X}{\mathcal{X}} 
\newcommand{\Orth}{\mathcal{O}}
\newcommand{\A}{\mathcal{A}} 
\newcommand{\fraka}{\mathfrak{a}} 
\newcommand{\frakb}{\mathfrak{b}} 
\newcommand{\frakc}{\varphi} 
\newcommand{\yy}{\mathsf{y}} 
\newcommand{\idem}{\varepsilon} 
\newcommand{\Tab}{\operatorname{Tab}} 
\newcommand{\BG}{\mathbf{B}} 
\newcommand{\bb}{\varnothing} 
\newcommand{\res}{\operatorname{res}} 
\newcommand{\Irr}{\operatorname{Irr}} 
\newcommand{\Wt}{\operatorname{Wt}} 
\newcommand{\trace}{\operatorname{trace}} 
\newcommand{\type}{\operatorname{type}} 
\newcommand{\gen}[1]{\langle #1 \rangle} 
\newcommand{\parm}{\updelta} 
\newcommand{\compose}{\circ} 

\newcommand{\qand}{\quad\hbox{and}\quad}
\newcommand{\F}{\mathbf{F}} 
\newcommand{\G}{\mathbf{G}} 
\newcommand{\Res}{\operatorname{Res}}
\newcommand{\Ind}{\operatorname{Ind}}
\newcommand{\into}{\hookrightarrow} 

\newcommand{\transpose}[1]{{#1}^{\emph{\textsf{T}}}}

    \newcommand{\demph}[1]{\textcolor{black}{\emph{#1}}}

\newcommand{\Part}[1]{
 \foreach \x [count=\s from 1] in {#1}{
 	{\ifnum\s=1
		\draw (0,\s-1)--(\x,\s-1); 
		\fi}
   \draw (0,\s) to (\x,\s);
   \foreach \y in {0, ..., \x} {\draw (\y,\s)--(\y,\s-1);}
 }}

 \newcommand{\fPart}[2]{
 \foreach \x [count=\s from 1] in {#1}{
  \filldraw[#2] (0,\s) -- (\x,\s)--(\x,\s-1)--(0,\s-1);
 }
 \foreach \x [count=\s from 1] in {#1}{
 	{\ifnum\s=1
		\draw (0,\s-1)--(\x,\s-1); 
		\fi}
   \draw (0,\s) to (\x,\s);
   \foreach \y in {0, ..., \x} {\draw (\y,\s)--(\y,\s-1);}
 }}

\def\UNIT{.25}

\newcommand{\tPART}[1]{ 
\begin{tikzpicture}[xscale=.8*\UNIT, yscale=-.8*\UNIT] 
	\Part{#1}
\end{tikzpicture}
}
\newcommand{\ttPART}[1]{ 
\begin{tikzpicture}[xscale=.5*\UNIT, yscale=-.5*\UNIT] 
	\Part{#1}
\end{tikzpicture}
}

\swapnumbers 
\newtheorem{thm}{Theorem}[section]
\newtheorem*{thm*}{Theorem}
\newtheorem{lem}[thm]{Lemma}
\newtheorem*{lem*}{Lemma}
\newtheorem{prop}[thm]{Proposition}
\newtheorem*{prop*}{Proposition}
\newtheorem{cor}[thm]{Corollary}
\newtheorem*{cor*}{Corollary}

\newtheorem*{conj*}{Conjecture}

\theoremstyle{definition}
\newtheorem{defn}[thm]{Definition}
\newtheorem*{defn*}{Definition}
\newtheorem{example}[thm]{Example}
\newtheorem*{example*}{Example}
\newtheorem{examples}[thm]{Examples}
\newtheorem*{examples*}{Examples}
\newtheorem{alg}[thm]{Algorithm}
\newtheorem*{alg*}{Algorithm}
\newtheorem{rmk}[thm]{Remark}
\newtheorem*{rmk*}{Remark}
\newtheorem{rmks}[thm]{Remarks}
\newtheorem*{rmks*}{Remarks}

\numberwithin{equation}{section} 
\allowdisplaybreaks
\renewcommand{\theenumi}{\alph{enumi}} 

\title[Canonical idempotents of multiplicity-free families]{Canonical
  idempotents of multiplicity-free families of algebras}

\author{Stephen Doty}
\author{Aaron Lauve}
\author{George H.~Seelinger}
\address{Department of Mathematics and Statistics, 
	Loyola University Chicago,
	Chicago, IL 60660 USA}%
\email{doty@math.luc.edu}%
\email{lauve@math.luc.edu}
\address{Department of Mathematics,
	University of Virginia,
	P.O. Box 400137,
	Charlottesville, VA 22904 USA} 
\email{ghs9ae@virginia.edu}

\date{\today} 

\begin{document}

\begin{abstract}\noindent
Any multiplicity-free family of finite dimensional algebras has a
\emph{canonical} complete set of pairwise orthogonal primitive
idempotents in each level.  We give various methods to compute these
idempotents. In the case of symmetric group algebras over a field of
characteristic zero, the set of canonical idempotents is precisely the
set of seminormal idempotents constructed by Young. As an example, we
calculate the canonical idempotents for semisimple Brauer algebras.
\end{abstract}

\maketitle

\section*{Introduction}\noindent
Given a finite dimensional unital associative algebra $\A$ over a
field $\Bbbk$, a fundamental problem is to find a \emph{partition of
  unity}, i.e., a complete set of pairwise orthogonal primitive
idempotents, in $\A$.  (This means finding a set $\{e_i\}_{i\in I}$ of
elements satisfying $\sum_i e_i =1$ and $e_ie_j = \delta_{ij} e_j$ for
$i, j\in I$, with $|I|$ maximal.)
The corresponding problem for the
center $Z(\A)$ is equally fundamental; in that case the partition is
unique. We study these two closely related problems under the
assumption that $\A$ is split semisimple; i.e., $\A$ is isomorphic to
a direct sum of matrix algebras over $\Bbbk$.

Our main results are for the special case where $\A = \A_n$ fits into
a multiplicity-free family $\{\A_n \mid n \ge 0 \}$ (see Definition
\ref{def:mfa}), which allows for induction on $n$. Group algebras of
symmetric groups serve as the primary motivating example. For a
multiplicity-free family $\{\A_n\}$, we find that:
\begin{enumerate}\renewcommand{\theenumi}{\arabic{enumi}} 
\item \label{item:canonical} There is a canonical partition of unity
  of $\A_n$ for all $n$ (see Proposition \ref{prop:idem-gen}). This
  fact is implicit in \cites{OV,VO} and explicit in
  \cite{Goodman-Graber:JM}; we feel it deserves to be more widely known.

\item \label{item:equivalent} The two problems (calculating the
  canonical partitions of unity in $\A_n$ and in $Z(\A_n)$ for all
  $n$) are equivalent.

\item Both problems can be solved recursively by ``Lagrange
  interpolation'' methods, in terms of the eigenvalues of a
  Jucys--Murphy sequence on a Gelfand--Tsetlin basis of the
  irreducible representations.

\item \label{item:P-lambda-mu} Both problems reduce to the computation
  of certain polynomials in the $n$th Jucys--Murphy element, for all
  $n$. The polynomials depend only on a pair $(\lambda, \mu)$ of
  isomorphism classes of irreducible representations, one for $\A_n$
  and the other for $\A_{n-1}$.
\end{enumerate}
Many of the results of the paper are straightforward extensions of
known results scattered through the literature.  Our approach is based
on the insights of Vershik and Okounkov \cites{OV, VO} for symmetric
group algebras; see also \cites{Dipper-James, Goodman-delaHarpe-Jones,
  Ram-Wenzl, Halverson-Ram-chars, Leduc-Ram, Ram-PLMS,
  Ram-second-chars, Diaconis-Greene, OP, Garsia, Mathas,
  CST-symmetric-gps, Goodman-Graber:JM} for related work. Probably
\cite{Goodman-Graber:JM} overlaps the most with this paper.

The general theory of Lagrange interpolation methods for
multiplicity-free families is presented in Sections
\ref{sec:mfa}--\ref{sec:jm}; this theory extends known results from
symmetric group algebras in characteristic zero to arbitrary
multiplicity-free families.  Examples of multiplicity-free families
abound in the literature (e.g.\ partition algebras, Temperley--Lieb
algebras, various families of Weyl groups and their associated Hecke
algebras, Birman--Murakami--Wenzl algebras) so these results should
have wide applicability. For many of these families, suitable
candidates for Jucys--Murphy sequences (in our sense) have been found,
which should bring all of items
\eqref{item:canonical}--\eqref{item:P-lambda-mu} above to bear on
their study. Due to space constraints, we treat only two illustrative
examples here: in Sections \ref{sec:symm-gp} and \ref{sec:Brauer} we
apply our methods to study the symmetric group algebras and Brauer
algebras, respectively.  Although we have chosen to avoid the language
of cellular algebras, in order to keep the exposition as elementary as
possible, readers interested in applying these results to other
diagram algebras would be well-advised to utilize the axiomatic
framework of \cite{Goodman-Graber:JM} and the related results of
\cite{Goodman-Graber:JBC}.

Appendix \ref{sec:pci} outlines an alternative method of computing the
partition of unity of $Z(\A)$ in characteristic zero, based on trace
characters instead of interpolation. This is valid without any
assumption that the split semisimple algebra $\A$ fits into a
multiplicity-free family; however, it requires inverting a possibly
large matrix.

\subsection*{Acknowledgments}\noindent
This project started as an undergraduate research project by the third
author, jointly mentored by the first two. The authors are grateful to
the Mulcahy Scholars Program of Loyola University Chicago for support.
Our work was greatly influenced by a seminar talk by Tony Giaquinto
\cite{Giaquinto}. We would also like to thank Stuart Martin and Peter
Tingley for useful conversations and advice, and the referee for
suggesting substantial improvements.

\section{Multiplicity-free families of algebras}\label{sec:mfa}
\noindent
Let $\Bbbk$ be a field and $\A$ an algebra over $\Bbbk$. All the
algebras considered in this paper are assumed to be finite
dimensional, semisimple, associative, unital, and split over
$\Bbbk$. Write $\Irr(\A)$ for the set of isomorphism classes of
irreducible left\footnote{\,We could just as well work with right
  modules, and will do so in Sections \ref{sec:symm-gp},
  \ref{sec:Brauer}.} $\A$-modules and $V^\lambda$ for a representative
of the class $\lambda\in\Irr(\A)$. That is, $[V^\lambda] = \lambda$.

The general Wedderburn--Artin theorem expresses $\A$ as a finite
direct sum of matrix algebras over division rings; our assumption that
$\A$ is split over $\Bbbk$ means that each of the division rings is
$\Bbbk$ (this is automatic if $\Bbbk$ is algebraically closed), so
\begin{equation}\label{eq:ss-dec}
  \A \, = \! \bigoplus_{\lambda \in \Irr(\A)} \!\idem(\lambda) \A \ 
  	\cong \bigoplus_{\lambda \in \Irr(\A)} \! \End_\Bbbk(V^\lambda).
\end{equation}
In the isomorphism \eqref{eq:ss-dec}, the central idempotent
$\idem(\lambda)\in\A$ acts as the identity in $\End_\Bbbk(V^\lambda)$
and zero in the other components, so $\{ \idem(\lambda) \mid \lambda
\in \Irr(\A)\}$ is the (unique) partition of unity of the center
$Z(\A)$.

The main objective of this paper is to study the situation where $\A =
\A_n$ fits into an infinite family of algebras satisfying the
following properties.

\begin{defn}\label{def:mfa}
  A family $\{\A_n \mid n \ge 0 \}$ of finite dimensional split
  semisimple algebras over a field $\Bbbk$ is a
  \demph{multiplicity-free family of algebras} if the following axioms
  hold:
  
  (a) (Triviality) $\A_0 \cong \Bbbk$.

  (b) (Embedding) For each $n$, there is a unity preserving algebra
  embedding $\A_n \into \A_{n+1}$.

  (c) (Branching) The restriction to $\A_{n-1}$ of an irreducible
  $\A_n$-module $V$ is isomorphic to a direct sum of pairwise
  non-isomorphic irreducible $\A_{n-1}$-modules.
\end{defn}

Whenever (c) above holds, we say that restriction from $\A_n$ to
$\A_{n-1}$ is multiplicity-free.  The following general criterion
characterizes this property.

\begin{prop}[\cite{VO}*{Prop.~1.4}]
  Restriction from $\A_n$ to $\A_{n-1}$ is multiplicity-free if and
  only if the centralizer algebra 
  \[
  Z(\A_{n-1}, \A_n) = \{ x \in \A_n \mid xy=yx, \text{ for all } y \in
  \A_{n-1} \}
  \] 
  is commutative.
\end{prop}

To ease notation, whenever we have a multiplicity-free family we write
$\Irr(n)$ short for $\Irr(\A_n)$.  Extending \cites{OV,VO}, we define
the \demph{branching graph} $\BG$ (or \demph{Bratteli diagram}) of the
given family to be the directed graph with vertices and edges as
follows:
\begin{itemize}
\item the vertices are the isomorphism classes $\bigsqcup_{n \ge 0} \Irr(n)$;
\item there is an edge $\mu \to \lambda$ from the vertex $\mu$ to the
  vertex $\lambda$ if and only if $V^\mu$ is isomorphic to a direct
  summand of the restriction of $V^\lambda$.
\end{itemize}
Given $\lambda \in \Irr(n)$, let $\Tab(\lambda)$ denote the set of
\demph{paths} in the branching graph starting from the unique element
$\bb\in\Irr(0)$ and terminating at $\lambda$.\footnote{\,The set
  $\Tab(\lambda)$ is analogous to the set of standard tableaux of
  shape $\lambda$ in the representation theory of symmetric groups.}
Concretely, an element of $\Tab(\lambda)$ has the form
\[
\T = (\lambda_0 \to \lambda_1 \to \lambda_2 \to \cdots \to
\lambda_{n-1} \to \lambda_n),
\]
where $\lambda_0 = \bb$ and $\lambda_n = \lambda$.
Set $\Tab(n) = \bigsqcup_{\lambda \in \Irr(n)} \Tab(\lambda)$.  We say
that $\T \in \Tab(n)$ is a path of \demph{length} $n$ (a path on $n+1$
vertices). We sometimes write $\T \mapsto \lambda$ to indicate that
$\T \in \Tab(\lambda)$. We also write $\overline\T$ for the path in
$\Tab(n{-}1)$ obtained from $\T$ by deleting its last edge,
$\lambda_{n-1}\to\lambda_n$.
 
We now describe how to use branching to produce bases of irreducible
modules.  Let $V$ be a given irreducible $\A_n$-module. By the
branching rule \ref{def:mfa}(c) and Schur's Lemma, the decomposition
\begin{equation}
  \res_{\A_{n-1}} V = \textstyle \bigoplus_{[W] \to [V]} W
\end{equation}
is canonical. 
Decomposing each $W$ on the right hand side upon restriction to
$\A_{n-2}$ and continuing inductively all the way down to $\A_0 \cong
\Bbbk$, we obtain a canonical decomposition
\begin{equation}\label{eq:res-A_1}
  \res_{\A_0} V = \textstyle \bigoplus_{\T} V_\T
\end{equation}
into irreducible $\A_0$-modules, which are the 1-dimensional subspaces
$V_\T$, where the index $\T$ runs over the set of $\T \in \Tab(n)$
terminating in $[V]$.  Note that the $\A_k$-submodule of $V$ generated
by $V_\T$ is isomorphic to $V^{\lambda_k} = \idem(\lambda_k) \cdots
\idem(\lambda_n)V$, where $\lambda_k$ is the $k$th vertex in the path
$\T$, for each $k = 0, 1, \dots, n-1, n$.  Choosing a nonzero vector
$v_\T \in V_\T$ for each $\T$ in $\Tab(n)$, we get a basis
\[
  \{ v_\T \mid \T \mapsto [V] \} 
\]
of each $V$, called the \demph{Gelfand--Tsetlin basis}; this idea goes
back to \cites{GT1,GT2}.  We note that the choice of $v_\T$ is
uniquely determined only up to a scalar multiple.

In what follows, an important role is played by the
\demph{Gelfand--Tsetlin} subalgebra $\X_n$ ($n\geq1$). Following
\cite{VO}, this is the subalgebra of $\A_n$ generated by the centers
\[
  Z(\A_1), Z(\A_2), \dots, Z(\A_n).
\]
It is easy to see that $\X_n$ is a commutative subalgebra of $\A_n$,
for all $n$.  Clearly $\X_n \subseteq \X_{n+1}$, for all $n$.  

\begin{defn}\label{def:idem-T}
  To each path $\T : \bb=\lambda_0 \to \lambda_1 \to \cdots \to
  \lambda_n$ of length $n$ in the branching graph, we associate a
  unique element $\idem_\T := \idem(\lambda_1) \idem(\lambda_2) \cdots
  \idem(\lambda_n)$ of the Gelfand--Tsetlin subalgebra $\X_n$.
\end{defn}

\begin{rmk}\label{rmk:idem-recursive} 
Equivalently, $\idem_\T$ can be defined recursively by:
\begin{equation*}
   \idem_\T = 
   \begin{cases}
      \idem_{\ov{\T}} \cdot \idem(\lambda_n) & \text{ if } n>0, \\
      1  & \text{ if } n = 0
   \end{cases}
\end{equation*}
in terms of the notation $\ov{\T}$ introduced above.
\end{rmk}

Given an irreducible module $V\cong V^\lambda$ for $\A_n$ and any $\T
\mapsto \lambda$, the element $\idem_\T\in\A_n$ is the projection
mapping $V$ onto $V_\T$. In \cite{VO}*{Prop.~1.1}, Vershik and
Okounkov use these canonical projections to prove the following
result.

\begin{prop}\label{prop:max-comm}
The Gelfand--Tsetlin algebra $\X_n$ is the algebra of all elements of
$\A_n$ that act diagonally on the Gelfand--Tsetlin basis $\{ v_\T \}$
for each irreducible $\A_n$-module $V$.  In particular, the algebra
$\X_n$ is a maximal commutative subalgebra of $\A_n$.
\end{prop}

\begin{proof}
  Suppose that $\T \mapsto \lambda\in \Irr(n)$. Since $\idem_\T$
  projects $V\cong V^\lambda$ onto its one-dimensional subspace
  $V_\T$, it follows that $\idem_\T$ sends $v_\T$ to itself. Also,
  $\idem_\T$ acts as zero on all $v_\TS$ such that $\TS \ne \T$. So
  with respect to the Gelfand--Tsetlin basis $\{ v_\T \}$ for $V$, the
  operators $\idem_\T$ are diagonal matrices. In view of
  \eqref{eq:ss-dec}, the algebra generated by
  $\{\idem_\T \mid \T \in \Tab(n)\}$ is a maximal commutative
  subalgebra of $\A_n$. Since $\X_n$ is commutative and contains this
  subalgebra, we have equality, which completes the proof.
\end{proof}

The following result did not explicitly appear in \cite{VO}, although
it is implicit in their setup. It provides an explicit and canonical
partition of unity in $\A_n$ for each $n$, in terms of the primitive
central idempotents.

\begin{prop}\label{prop:idem-gen}
  The set $\{ \idem_\T \mid \T \in \Tab(n) \}$ is a family of pairwise
  orthogonal primitive idempotents in $\A_n$ that sums to $1$ (the
  unit in $\A_n$).  It is also a $\Bbbk$-basis for the
  Gelfand--Tsetlin subalgebra $\X_n$.
\end{prop}

\begin{proof}
It is clear from Definition \ref{def:idem-T} that $\idem_\T =
\idem(\lambda_1) \cdots \idem(\lambda_n)$ is idempotent for any $\T$,
since its factors commute. The commutativity of the factors is also
used to check that $\idem_\T \idem_{\T'} = 0$ if either
\[
  \T \mapsto \lambda \ \text{and} \ \T' \mapsto \lambda'
  \quad\text{with}\quad \lambda \neq \lambda'
\]
or
\[
  \T \mapsto \lambda \ \text{and} \ \T' \mapsto \lambda
  \quad\text{with}\quad \T \ne \T'.
\]
So the idempotents are pairwise orthogonal. 

For any $\T \mapsto \lambda$, $\idem_\T$ acts as one on $V^\lambda_\T$
and zero on all $V^\lambda_\TS$, for $\TS \ne \T$. Since
$V^\lambda = \bigoplus_{\T \mapsto \lambda} V^\lambda_\T$, it follows
that $\sum_{\T \mapsto \lambda} \idem_\T$ and $\idem(\lambda)$ both
act as one on $V^\lambda$. Furthermore, both act as zero on $V^\mu$,
for each $\lambda \ne \mu\in\Irr(n)$.  This shows that $\sum_{\T
  \mapsto \lambda} \idem_\T = \idem(\lambda)$. It follows that
$\sum_{\T \in \Tab(n)} \idem_\T = \sum_{\lambda \in \Irr(n))}
\idem(\lambda) = 1$.

Finally, the various $\idem_\T$ are primitive
since we have precisely the right number, namely
$\sum_{\lambda \in \Irr(n)} \dim_\Bbbk V^\lambda = |\Tab(n)|$.

The last claim in the proposition follows from the proof of
Proposition \ref{prop:max-comm}, since the $\idem_\T$ are linearly
independent and $\dim_\Bbbk \X_n = \sum_{\lambda \in \Irr(n)}
\dim_\Bbbk V^\lambda$.
\end{proof}


\begin{cor}\label{cor:idem-characterization}
  The canonical idempotents $\{ \idem_\T \mid \T \in \Tab(n) \}$
  satisfy the following properties:
  \begin{enumerate}
  \item $\sum_{\T \mapsto \lambda} \idem_\T = \idem(\lambda)$, for all
    $\lambda \in \Irr(n)$.
  \item $\idem_{\ov{\T}} \idem_\T = \idem_\T$, for all $\T \mapsto
    \lambda$, $\lambda \in \Irr(n)$. 
  \end{enumerate}
  Furthermore, $\{ \idem_\T \mid \T \in \Tab(n), n \ge 0 \}$ is the
  unique set of pairwise orthogonal idempotents satisfying these two
  properties.
\end{cor}

\begin{proof}
  Property (a) was proved already in the proof of the previous
  proposition. Property (b) follows immediately from the definition of
  $\idem_\T$ and the definition of $\ov{\T}$. 

  Suppose that $\{ g_\T \mid \T \in \Tab(n), n\ge 0 \}$ is another set
  such that for each fixed $n$, the set $\{g_\T \mid \T \in \Tab(n)\}$
  is a set of pairwise orthogonal idempotents in $\A_n$ satisfying
  properties (a) and (b). For the unique path $\bb$ of length $0$, we
  have $g_\bb = \idem_\bb = 1$. Proceeding by induction on $n$,
  suppose that $n>0$ is fixed and assume that $g_\TS = \idem_\TS$ for
  all paths $\TS$ of length strictly less than $n$. Then for
  $\T \in \Tab(n)$ with $\T \mapsto \lambda$, we have
  \begin{align*}
    \idem_\T &= \idem_{\ov{\T}} \idem(\lambda) = g_{\ov{\T}} \sum_{\TS
               \mapsto \lambda} g_\TS = g_{\ov{\T}} \sum_{\TS
               \mapsto \lambda} g_{\ov{\TS}} g_\TS\\
    & = \sum_{\TS \mapsto \lambda} g_{\ov{\T}} g_{\ov{\TS}} g_\TS =
      \sum_{\TS \mapsto \lambda} \delta_{\ov{\T},\ov{\TS}}\,
      g_{\ov{\TS}} g_\TS = g_{\ov{\T}} g_\T = g_\T.
  \end{align*}
  Note that the penultimate equality above is valid because $\T$ is
  the only path of shape $\lambda$ whose restriction of length $n-1$
  is $\ov{\T}$. 
\end{proof}

It is illuminating to introduce a \emph{global Gelfand--Tsetlin basis}
for $\A_n$ at this point. 

Fix a Gelfand--Tsetlin basis $\{v_\T \mid \T \mapsto \lambda \}$ for
each irreducible $V^\lambda$, $\lambda \in \Irr(n)$.  We may identify
the algebra $\End_\Bbbk(V^\lambda)$ with the matrix algebra
$\text{Mat}_{\dim V^\lambda}(\Bbbk)$ by means of the basis.  Let
$\frakc^\lambda_{\TS,\T}$ be the $\Bbbk$-linear endomorphism of
$V^\lambda$ mapping $v_\T$ to $v_\TS$ and all other $v_{\T'}$ to
$0$. The set
\[
  \{ \frakc^\lambda_{\TS,\T} \mid \TS, \T \mapsto \lambda \}
\]
is a basis of $\End_{\Bbbk}(V^\lambda)$; under the identification
$\End_\Bbbk(V^\lambda) \cong \text{Mat}_{\dim V^\lambda}(\Bbbk)$, it
corresponds to the basis of matrix units.  The desired global
Gelfand--Tsetlin basis of $\A_n$ under the isomorphism
\eqref{eq:ss-dec} is the disjoint union
\begin{equation}\label{eq:global-GZ}
  \bigsqcup_{\lambda \in\, \Irr(n)} \, \{ \frakc^{\lambda}_{\TS,\T} \mid \TS, \T
  \mapsto \lambda \}.
\end{equation}
This basis is uniquely determined by the choice of Gelfand--Tsetlin
basis $\{ v_\T \}$ for each $V^\lambda \in \Irr(n)$, but it depends on
those choices. Note that
$\frakc^\lambda_{\TS,\T} \cdot \frakc^\mu_{\TS',\T'} = 0$ for
$\lambda \neq \mu$; this follows from the equality
$\Hom_{\A_n}(V^\lambda,V^\mu) = 0$, which is true by Schur's
Lemma. Hence the basis \eqref{eq:global-GZ} satisfies
\begin{equation}\label{eq:mult-props}
  \frakc^\lambda_{\TS,\T} \cdot \frakc^\mu_{\TS',\T'} = \delta_{\lambda,\mu} \,
  \delta_{\T,\TS'} \, \frakc^\lambda_{\TS,\T'}
\end{equation}
where $\delta$ is the usual Kronecker delta.  In particular, each
$\frakc^\lambda_{\T,\T}$ is an idempotent. We note that
\eqref{eq:mult-props} implies that the basis \eqref{eq:global-GZ} is a
cellular basis in the sense of \cite{Graham-Lehrer}.

The above allows us to model the algebra $\A_n$ isomorphically as the
matrix algebra consisting of all $N \times N$ block diagonal matrices,
where $N = \sum_{\lambda \in \, \Irr(n)} \dim_\Bbbk V^\lambda$, such
that the block indexed by each $\lambda$ is a full matrix algebra of
$d \times d$ matrices over $\Bbbk$, where $d = \dim V^\lambda$. Of
course, since the Gelfand--Tsetlin bases of the irreducible
representations are unique only up to choice of scalars, this model
depends on those choices. However, being products of the unique
central idempotents, the $\idem_\T$ themselves are independent of the
choices.

\begin{cor}\label{cor:endo-img}
  Under the identification $\A_n \cong \bigoplus_{\lambda \in
    \Irr(n)} \End_\Bbbk(V^\lambda)$ of \eqref{eq:ss-dec}, the
  primitive central idempotent $\idem(\lambda)$ corresponding to any
  $\lambda \in \Irr(n)$ satisfies the identity
  \[
  \idem(\lambda) = \sum_{\T \mapsto \lambda} \frakc^\lambda_{\T,\T}.
  \]
  Likewise, for any path $\T \mapsto \lambda$ in $\BG$ we have the
  identity
  \[
   \idem_\T = \frakc^\lambda_{\T,\T} .
  \]
\end{cor}

\begin{proof}
To prove that $\idem(\lambda) = \sum_{\T \mapsto \lambda} \frakc^\lambda_{\T,\T}$
observe that both sides act as one on $V^\lambda$ and as zero on all other
irreducibles $W\not\cong V^\lambda$. Similarly, the equality $\idem_\T =
\frakc^\lambda_{\T,\T}$ follows from the fact that both sides act the same
on all $V^\mu$ ($\mu \in \Irr(n)$).
\end{proof}

\begin{rmks}\label{rmk:eigenvalues}
  (a) Let $\X_n$ be the maximal commutative subalgebra of $\A_n$
  defined above. By Proposition \ref{prop:idem-gen}, it is spanned by
  the idempotents $\idem_\T$. Then it is clear from
  \eqref{eq:mult-props} and Corollary \ref{cor:endo-img} that the
  global Gelfand--Tsetlin basis $\{ \frakc^\lambda_{\TS,\T} \}$ is a
  basis consisting of simultaneous (left or right) eigenvectors for
  the action of $\X_n$ by left or right multiplication. To be
  explicit: an arbitrary element $\sum_\U c_\U\, \idem_\U$ of $\X_n$
  acts on $\frakc^\lambda_{\TS,\T}$ by left multiplication as the
  scalar $c_\TS$ and by right multiplication as the scalar $c_\T$.

  (b) Similarly, as already noted in Proposition \ref{prop:max-comm},
  the basis $\{ v_\T : \T \mapsto \lambda \}$ of $V^\lambda$ is a
  basis of simultaneous eigenvectors for the action of $\X_n$. To be
  explicit, the element $\sum_\T c_\T\, \idem_\T$ as above acts as
  $c_\T$ on the basis element $v_\T$, for each $\T$.

  (c) The decomposition $\A_n = \bigoplus_{\T \in \Irr(n)} \A_n
  \idem_\T$, which is a decomposition of $\A_n$ into a direct sum of
  irreducible left ideals, is actually a ``weight space''
  decomposition for the action of $\X_n$ by right multiplication, in
  the sense that each element of $\A_n \idem_\T$ is an eigenvector
  for the right action of an arbitrary element $\sum_{\U \in \Irr(n)}
  c_\U\, \idem_\U$ of $\X_n$, of eigenvalue $c_\T$. A similar remark,
  with left and right interchanged, holds for the decomposition $\A_n
  = \bigoplus_{\T \in \Irr(n)} \idem_\T\A_n$.
\end{rmks}

Thus, we see that in some sense the role of the Gelfand--Tsetlin
algebra $\X_n$ in the theory of multiplicity-free families is
analogous to that of a Cartan subalgebra in the theory of Lie
algebras.

\section{Central idempotents via interpolation}
\label{sec:centrals-via-int}
\noindent
The primitive central idempotents can be computed by a type of
Lagrange interpolation, provided that a generator of the center is
available. This applies to an arbitrary split semisimple finite
dimensional algebra $\A$, so we temporarily drop the assumption that
the algebra fits into a multiplicity-free family.

Note that $\{ \idem(\lambda) \mid \lambda\in\Irr(\A)\}$ is a basis for
the center $Z(\A)$. So any element $z \in Z(\A)$ is uniquely
expressible in the form
\[
    z = \sum_{\lambda \in \Irr(\A)} a_\lambda\, \idem(\lambda).
\]
It follows that $z\cdot \idem(\lambda) = a_\lambda \, \idem(\lambda)$
for all $\lambda\in\Irr(\A)$. Call the tuple
$\bigl(a_\lambda\bigr)_{\lambda \in \Irr(\A)}$ the
\demph{(eigen)spectrum} of $z$. A spectrum is \demph{simple} if it has
no repeated entries.

\begin{lem} \label{lem:spectrum}
(a) An element $z \in Z(\A)$ generates $Z(\A)$ if and only if its
  spectrum is simple.

(b) If $\Bbbk$ has at least as many elements as $|\Irr(\A)|$,
  the center $Z(\A)$ is generated (as an algebra) by a single element.
\end{lem}

\begin{proof}
(a) Regarded as a linear operator on $Z(\A)$ by multiplication, the
  element $z$ is diagonal with respect to 
  the basis $\{ \idem(\lambda) \mid \lambda\in\Irr(\A)\}$. 
  Let $S = \{ a_\lambda
  \mid \lambda \in \Irr(\A) \}$ be the set of \emph{distinct}
  eigenvalues of $z$. The minimal polynomial of $z$ is $\prod_{a \in
    S} (z- a)$.  Let $m = |\Irr(\A)| = \dim_\Bbbk Z(\A)$. Clearly, the
  element $z$ generates $Z(\A)$ if and only if the set $\{1, z, \dots,
  z^{m-1} \}$ is linearly independent. This is true if and only if the
  minimal polynomial of $z$ has degree $m$. So $z$ generates $Z(\A)$
  if and only if it has simple spectrum.

(b) Choose $m$ distinct elements of $\Bbbk$, say $a_1, \dots
  a_m$. Choose any enumeration $\lambda_1, \dots, \lambda_m$ of the
  elements of $\Irr(\A)$. Then $z = a_1 \idem(\lambda_1) + \cdots + a_m
  \idem(\lambda_m)$ has simple spectrum, hence generates $Z(\A)$. 
\end{proof}

\begin{rmks}
  (a) If $z = \sum_\lambda a_\lambda \,\idem(\lambda)$ is a generator of
  $Z(\A)$ then the change of basis matrix expressing the powers $1, z,
  \dots, z^{m-1}$ in terms of the idempotents $\idem(\lambda)$ is a
  Vandermonde matrix in the $a_\lambda$'s.

  (b) If the field $\Bbbk$ is large compared to $m = |\Irr(\A)|$,
  there are many generators of $Z(\A)$. In fact, if $\Bbbk$ is a
  finite field of $q$ elements, then the probability $P(q)$ that a
  randomly chosen element of $Z(\A)$ actually generates the center
  is
\[
P(q) = \frac{q(q-1)\cdots (q-m+1)}{q^m} =
\bigg(1-\frac{1}{q}\bigg)\cdots \bigg(1-\frac{m-1}{q}\bigg),
\]
Evidently, $\lim_{q \to \infty} P(q) = 1$.
\end{rmks}

The lemma leads immediately to an interpolation formula for the
$\idem(\lambda)$, provided that one can find a generator and compute
its spectrum.

\begin{prop}\label{prop:z-interpolation}
  Suppose that $z$ is a generator of $Z(\A)$, with spectrum
$\bigl(a_\lambda \mid \lambda \in \Irr(\A)\bigr)$. Then the polynomial
  \[
     Q_\lambda(z) = \prod_{\mu \in \Irr(\A): \, \mu \ne \lambda} \frac{z -
       a_\mu}{a_\lambda - a_\mu}
  \]
is equal to $\idem(\lambda)$, for each $\lambda \in \Irr(\A)$. 
\end{prop}

\begin{proof}
This is immediate from the fact that $\prod_{\mu \in \Irr(\A)} (z -
a_\mu) = 0$, which implies that 
\[
  Q_\lambda(z) \cdot \idem(\mu) = \delta_{\lambda, \mu} \, \idem(\lambda). 
\]
Hence $Q_\lambda(z) = Q_\lambda(z) \cdot 1 = \sum_{\mu \in \Irr(\A)}
Q_\lambda(z) \cdot \idem(\mu) = \idem(\lambda)$. 
\end{proof}

The formula in Proposition \ref{prop:z-interpolation} is useful only
if we have a way of retrieving $z$'s spectrum without already knowing
the central idempotents. At least in characteristic zero, this can be
done whenever the irreducible trace characters are known.

\begin{prop}[\cite{Ram-PLMS}, Lemma~1.9]\label{prop:char}
  Let $\chi^\lambda$ be the trace character of $V^\lambda$ for any
  $\lambda \in \Irr(\A)$. Suppose that $\Bbbk$ has characteristic
  zero. Writing $z = \sum_\lambda a_\lambda\, \idem(\lambda)$, we have
  $a_\lambda = \frac{\chi^\lambda(z)}{\chi^\lambda(1)}$.
\end{prop}

\begin{proof}
  For any $v \in V^\lambda$, $z$ acts as $a_\lambda$; i.e.,
  $z \cdot v = a(\lambda)\, v$, so $\trace(z)$ on $V^\lambda$ is equal
  to $(\dim_\Bbbk V^\lambda)\, a_\lambda$.  In other words,
  $\chi^\lambda(z) = \chi^\lambda(1) \, a_\lambda$.
\end{proof}

The above analysis leads to a probabilistic algorithm for computing
the primitive central idempotents.

\begin{alg}
  Suppose that $\Bbbk$ has characteristic zero. Then to
  compute all the central idempotents $\idem(\lambda)$,
  \begin{enumerate}
  \item Pick a random $z\in Z(\A)$ and compute its spectrum (using
    Proposition \ref{prop:char} or otherwise). If the spectrum is not
    simple, try again.

    \item Once a generator $z$ with simple spectrum is found, use
      Proposition \ref{prop:z-interpolation} to compute the
      $\idem(\lambda)$ for all $\lambda \in \Irr(\A)$.
  \end{enumerate}
\end{alg}

If this can be carried out, the formulas thus obtained will express
the $\idem(\lambda)$ in terms of polynomial expressions in some random
central element. One would usually prefer to have expressions for the
$\idem(\lambda)$ in terms of elements that are understood in some
explicit way. At the least, one would prefer to understand how the
chosen central generator $z$ interacts with some set of standard
generators for the algebra $\A$.

\begin{example}\label{ex:simple-zn}
  Let $\A = \mathcal{H}_q(n)$ be the Iwahori--Hecke algebra
  corresponding to the symmetric group $\Sym_n$, over a field $\Bbbk$
  such that $0 \ne q \in \Bbbk$, and assume that $\mathcal{H}_q(n)$ is
  split semisimple. In \cite{Dipper-James}, certain $q$-analogues of
  the original Jucys--Murphy elements in $\Bbbk\Sym_n$ were
  constructed in $\mathcal{H}_q(n)$. As pointed out in
  \cite{OP}*{\S8.1}, their sum $Z_n$ has simple spectrum, hence is a
  generator of the center $Z(\mathcal{H}_q(n))$. Thus the formula in
  Proposition \ref{prop:z-interpolation} computes the primitive
  central idempotents $\idem(\lambda) \in \mathcal{H}_q(n)$ for each
  $\lambda \vdash n$.
\end{example}

\begin{example}\label{ex:non-simple-zn}
  Let $\A = \Bbbk \Sym_n$ be the group algebra of a symmetric group
  over a field $\Bbbk$ of characteristic zero. Let $z_n$ be the formal
  sum of all the transpositions in $\Sym_n$, regarded as an element of
  $\Bbbk\Sym_n$. This is precisely the element obtained from the
  element $Z_n$ in the previous example, if $q$ is specialized to $1$.
  The central element $z_n$ generates $Z(\Bbbk \Sym_n)$ for
  $n = 2, 3, 4$, and $5$, but for $n=6$ it fails to do so. The
  eigenvalues of $z_6$ on the irreducible modules indexed by partitions
  $(4,1^2)$ and $(3,3)$ coincide. Likewise for $(3,1^3)$ and
  $(2^3)$. At this writing, we do not know of any satisfactory uniform
  choice of elements $z_n \in \Bbbk \Sym_n$ generating the respective
  centers. This seems to be an interesting open problem.
\end{example}


To conclude this section, we mention an alternative approach to
computing the primitive central idempotents for $\A$. Recall that if
$\A=\C G$ for a finite group $G$, Frobenius gave a formula for
$\idem(\lambda)$ in terms of the simple character $\chi_G^\lambda$.
This result was extended to split semisimple finite dimensional
algebras in characteristic zero by Kilmoyer; however, it involves
inverting a $(\dim\A)\times(\dim\A)$ matrix. See Appendix
\ref{sec:pci} for a brief exposition.

\section{Generalized Jucys--Murphy sequences}\label{sec:jm}
\noindent
Now we return to the study of multiplicity-free families $\{ \A_n
\}_{n \ge 0}$ and the problem of computing the canonical idempotents
$\{\idem_\T \mid \T\in\Tab(n)\}$, which form a basis of the
(commutative) Gelfand--Tsetlin subalgebra $\X_n$ ($n\geq
0$). Extracting key elements from the work of Jucys and Murphy, we
show how a carefully selected sequence of elements (one from each
$\X_n$) can be used to effectively solve this problem.

Before we begin, we apply the results of the previous section to this
end. Given a sequence $(z_n\in \X_n \mid n\geq1)$ of center-generating
elements, i.e., elements satisfying $\langle z_n\rangle = Z(\A_n)$, we
reach the $\idem_\T$ in two steps:
\begin{itemize}
\item[(i)] use Propositions \ref{prop:z-interpolation} and
  \ref{prop:char} to compute the $\idem(\lambda)$;
\item[(ii)] use Definition \ref{def:idem-T} to compute the $\idem_\T$.
\end{itemize}
So the problem is solved, provided one can find a sequence of
center-generating elements.  However, as noted in Example
\ref{ex:non-simple-zn}, even for the family of group algebras of
symmetric groups, such a sequence is not known.

Murphy \cite{Murphy-81} found a non center-generating sequence of
elements---known independently to Young \cite{Young} and Jucys
\cite{Jucys}---and applied them to give a new construction of Young's
seminormal form of symmetric group algebras. In recent years,
analogues of such elements have been found in a number of other
multiplicity-free families. The two key properties of the
Young--Jucys--Murphys elements are abstracted in the next
definition. But first, some notation.

From the definition of $\X_n$ we have a sequence of inclusions
\begin{equation}\label{eq:incl-X}
\X_1 \subset \cdots \subset \X_{n-1} \subset \X_n.
\end{equation}
Now given $J_1 \in \X_1$, $J_2 \in \X_2, \dots, J_n \in \X_n$, the
inclusions \eqref{eq:incl-X} imply that $J_1, J_2, \dots, J_n \in
\X_n$. Since $\{ \idem_\T \}$ is a basis of $\X_n$, we have scalars
$c_\T(k) \in \Bbbk$, for each $k = 1, \dots, n$, such that
\begin{equation}\label{eq:J-sum}
  J_k = \textstyle \sum_{\T \in \Tab(n)} c_\T(k)\, \idem_\T.
\end{equation}
In this way we associate an $n$-tuple $c_\T$ to each $\T\in\Tab(n)$, 
\begin{equation}
  c_\T = \bigl(c_\T(1), \dots, c_\T(n)\bigr) \in \Bbbk^n.
\end{equation}
We call this $n$-tuple the
\demph{$\T$-content} for the sequence $J_1, \dots, J_n$.

\begin{defn}\label{def:JM-defn}
Let $(J_n \mid n \in \N)$ be a sequence of elements such that $J_n \in
\X_n$ for each $n$. We say that the sequence is:
\begin{enumerate}
\item 
  \emph{additively central}\,\footnote{\,In some multiplicity-free families,
    one can find \emph{multiplicatively central} sequences. These
    sequences, which were considered in \cite{Goodman-Graber:JM}, have
    the property that the partial product $J_1 J_2 \cdots J_n$ belongs
    to $Z(\A_n)$, for all $n \in \N$, and furthermore that it acts as a
    nonzero scalar on each $V^\lambda$, $\lambda \in \Irr(n)$. The
    results in this section are equally valid in the multiplicative
    case, modulo a few adjustments that we leave to the reader.} if
  the $n$th partial sum $J_1 + \cdots + J_{n-1} + J_n$ belongs to
  $Z(\A_n)$, for all $n \in \N$; and
\item 
  \emph{separating} if $\X_n = \gen{J_1, J_2, \dots, J_n}$, for all
  $n \in \N$.
\end{enumerate}
The sequence is a 
\emph{Jucys--Murphy sequence} (JM-sequence for short) if it is both 
additively central and separating. 
\end{defn}

To explain our terminology, we mention that additively central
sequences allow for ease of computation of the content vectors
(Proposition \ref{prop:central-J_n-property}), while separating
sequences allow content vectors to distinguish different paths
$\TS,\T\in\Tab(n)$ (Proposition
\ref{prop:separation-generation-equivalent}).

\begin{prop}\label{prop:existence-JM}
  JM-sequences in multiplicity free families always exist, provided
  that the underlying field is infinite.
\end{prop}

\begin{proof}
Let $(z_n)_{n \in \N}$ be a center-generating sequence, which exists by
Lemma 2.1(b).  Then putting $J_n := z_n - z_{n-1}$ (and stipulating
that $z_0=0$), it is easy to check that $(J_n)_{n \in \N}$ is a sequence
that is additively central. Assuming inductively that
$\X_{n-1}=\gen{J_1,\ldots,J_{n-1}}$, we have
\[
\X_n = \gen{\X_{n-1},Z(\A_n)} = 
\gen{\X_{n-1},z_n} = 
\gen{\X_{n-1},J_n} = 
\gen{J_1,\ldots,J_n}.
\]
This shows that $(J_n)_{n \in \N}$ is also separating.
\end{proof}
We next investigate the independent notions of additively central and
separating sequences before returning to JM-sequences for our main
result (Theorem \ref{thm:central-P-recursion}).

Let $(J_k)_{k \in \N}$ be an additively central sequence.  As $z_n =
\sum_{k=1}^n J_k \in Z(\A_n)$, it acts as some scalar $a_\lambda$ on
any irreducible representation $V^\lambda$, for $\lambda \in
\Irr(n)$. Similarly, $z_{n-1} = \sum_{k=1}^{n-1} J_k \in \A_{n-1}$, so
$z_{n-1}$ acts as a scalar $a_{\mu}$ on any $V^\mu$, for $\mu \in
\Irr(n{-}1)$. The next proposition shows how these scalars determine
the $n$th eigenvalue $c_\T(n)$. (Recall from Section \ref{sec:mfa} the
construction of $\ov{\T} \in \Tab(n{-}1)$ for any $\T\in\Tab(n)$.)

\begin{prop}\label{prop:central-J_n-property}
  Let $(J_k \mid k \in \N)$ be an additively central sequence of elements
  in a given multiplicity-free family $\{\A_n \mid n \ge 0 \}$.  For
  any $n$, let $\lambda \in \Irr(n)$. For any $\T \mapsto \lambda$ we
  have $c_\T(n) = a_{\lambda} - a_{\mu}$, where $\ov{\T} \mapsto \mu$.
\end{prop}

\begin{proof}
By hypothesis, we have $J_n = z_n - z_{n-1}$ for all $n \in \N$ (where
we set $z_0=0$). Then $z_n$ acts by right multiplcation as $a_\lambda$
on $v_\T$. By hypothesis we have $v_\T = v_{\ov{\T}}$ since $\T$ has
the form
\[
  \T = (\lambda_0 \to \lambda_1 \to \cdots \to \lambda_{n-1} \to
  \lambda_n),
\]
where $\lambda_n = \lambda$ and $\lambda_{n-1} = \mu$. So the element
$z_{n-1}$ acts as the scalar $a_\mu$ on
$v_\T$, and thus by linearity $J_n = z_n - z_{n-1}$ acts on $v_\T$ as
the scalar $a_\lambda - a_\mu$. The result is proved.
\end{proof}

\begin{rmk}\label{rmk:labels}
Note that Proposition \ref{prop:central-J_n-property} says that the
eigenvalue $c_\T(n)$ depends only on the last edge $\lambda_{n-1} \to
\lambda_n$ of the path $\T$ in the branching graph. So, whenever we
have an additively central sequence in our multiplicity-free family,
it makes sense to label each edge in level $n$ of the branching graph
by its corresponding eigenvalue $c_\T(n)$.  Figure \ref{Bratteli-Sym}
gives an example of such a labeled graph.
\end{rmk}

Let $(J_n)_{n \in \N}$ be any sequence with $J_n \in \X_n$ for each
$n$. In the proof of the following result, which characterizes
separating sequences, we will focus on the $\T$-contents one
coordinate at a time. Put $\Wt(k) := \{c_\T(k) \mid \T \in
\Tab(n)\}$. This is the set of eigenvalues of the operator $J_k$
acting on the various $\idem_\T$.

\begin{prop}\label{prop:separation-generation-equivalent}
  Given a sequence $J_1 \in \X_1$, $J_2 \in \X_2, \dots, J_n \in \X_n$, and 
  corresponding content vectors $c_\T$,
  the following are equivalent.
  \begin{enumerate}
  \item For all $\TS,\T\in\Tab(n)$, $\TS=\T \iff c_\TS = c_\T$.
  \item $\gen{J_1, \dots, J_n} = \X_n$.
  \end{enumerate}
\end{prop}

\begin{proof}
(a) $\Rightarrow$ (b).  We aim to show that $\idem_\T \in \gen{J_1,
    \dots, J_n}$ for each $\T \in \Tab(n)$, which would complete this
  half of the proof. To this end, note that the polynomial $E_\T(J) =
  E_\T(J_1, \dots, J_n)$ defined by
\[
E_\T(J) = \prod_{k=1}^n \prod_{\substack{c \in \Wt(k) \\ c \neq
    c_\T(k)}} \frac{J_k - c}{c_\T(k) - c}
\]
is well-defined as an operator on $\X_n$ (acting by
multiplication). Given $\TS \ne \T$, $E_\T(J)$ acts on the basis
element $\idem_\TS$ as
\[
E_\T(J) \,\idem_\TS = \prod_{k=1}^n \prod_{\substack{c\in\Wt(k)\\
    c\neq c_\T(k)}} \frac{(c_\TS(k)-c)\idem_\TS}{c_\T(k)-c} = 0,
\]
since $c_\TS(k)$ is among the $c\in\Wt(k)$ and $c_\TS(k) \ne c_\T(k)$
for at least one $k$, by (a).  A similar calculation shows that
$E_\T(J)\, \idem_\T = \idem_\T$.  Hence
\[
E_\T(J) \, = \, E_\T(J) \,\cdot\!\!\! \sum_{\TS\in\Tab(n)} \!\!\!\idem_\TS \,\, 
= \sum_{\TS\in\Tab(n)}  \!\!\!E_\T(J)\, \idem_\TS \, = \, \idem_\T,
\]
and hence $\idem_\T\in \gen{J_1,\ldots,J_n}$, as required.

(b) $\Rightarrow$ (a). Assume that $\TS,\T$ are paths such that $c_\TS
= c_\T$. We show that $\TS=\T$. Note that for any polynomial
$F(J_1,\ldots, J_n) \in \gen{J_1, \dots, J_n}$, $F$ acts on $v_\T$ and
$v_\TS$ by the same scalar, namely $F(c_\T(1), \dots, c_\T(n)) =
F(c_\TS(1), \dots, c_\TS(n))$.  Under the hypothesis (b), i.e.,
$\gen{J_1, \dots, J_n} = \X_n$, we know that $\idem_\T$ is such a
polynomial. Since
\[
  \idem_\T\cdot v_\T = v_\T
  \qand
  \idem_\T\cdot v_\TS = \delta_{\T,\TS} v_\TS,
\]
we must have $\delta_{\T,\TS}=1$. So $\TS = \T$. The converse
implication, that $\TS = \T$ implies $c_\TS = c_\T$, is trivial.
\end{proof}

Note that if $\X_n = \gen{J_1, \dots, J_n}$ then the above proof gives
the explicit formula
\begin{equation}\label{eq:idem_T-Lagrange}
  \idem_\T = E_\T(J) = \prod_{k=1}^n \prod_{\substack{c \in \Wt(k) \\ c \neq
      c_\T(k)}} \frac{J_k - c}{c_\T(k) - c} \,,
\end{equation}
expressing the canonical idempotents $\idem_\T$ in terms of the
separating sequence.  (A similar interpolation formula appears in
\cite{Murphy-81} in the context of symmetric group algebras.)
We find another interpolating polynomial for the $\idem_\T$, having
significantly lower degree than this one, in Theorem
\ref{thm:idem-T-recursion}.

\begin{prop} \label{prop:GJM-separation}
Let $(J_k \mid k \in \N)$ be a separating sequence in the multiplicity-free
family $\{\A_k \mid k \ge 0 \}$. Suppose that $\TS, \T \in
\Tab(n)$. Then
\begin{enumerate}
  \item $c_\T(k) = c_{\ov{\T}}(k)$ for all $k < n$.

  \item If $\ov{\TS} = \ov{\T}$ but $\TS \ne \T$ then $c_\TS(n) \ne
      c_\T(n)$.
\end{enumerate}
\end{prop}

\begin{proof}
(a) This follows from \eqref{eq:J-sum} and its analog for
  $\Tab(n{-}1)$, and the recursive description $\idem_\T =
  \idem_{\ov{\T}} \, \idem(\lambda)$ in Remark
  \ref{rmk:idem-recursive}. Specifically, we have
\[
c_\T(k)\idem_\T = J_k \cdot \idem_\T = J_k \cdot \idem_{\ov{\T}} \,
\idem(\lambda) = c_{\ov{\T}}(k)\, \idem_{\ov{\T}} \, \idem(\lambda) =
c_{\ov{\T}}(k)\, \idem_\T
\]
for any $k < n$.

(b) Since $\ov{\TS} = \ov{\T}$, it follows from part (a) that $c_\TS(k)
= c_\T(k)$ for all $k < n$.  If $\TS\neq \T$ and $c_\TS(n)=c_\T(n)$,
then $c_\TS = c_\T$ and we reach a contradiction with Proposition
\ref{prop:separation-generation-equivalent}.
\end{proof}

Proposition \ref{prop:GJM-separation} implies that the following
polynomial is well-defined. 

\begin{defn}\label{def:P_T}
Let $(J_n \mid n \in \N)$ be a separating sequence in a
multiplicity-free family. For any $\T \in \Tab(n)$, put
\[
  P_\T(J_n) \ := \prod_{\substack{\TS \in \Tab(n) \\ \TS \ne \T, \, \ov{\TS} =
    \ov{\T}}} \frac{J_n - c_{\TS}(n)}{c_\T(n) - c_\TS(n)}.
\]
\end{defn}

The next theorem shows that these polynomials can be used to
recursively compute the idempotents $\idem_\T$. It extends
\cite{Garsia}*{Theorems 3.4, 3.5} from symmetric group algebras to
multiplicity-free families. First, we record some basic properties of
the $P_\T(J_n)$.  Given $\TS,\T \in \Tab(n)$, we have
\begin{align}
  P_\T(J_n) \cdot \idem_\T &= \idem_\T \label{eq:P_T-one} \\ P_\T(J_n)
  \cdot \idem_{\TS} &= 0 \text{ if } \TS \ne \T \text{ but } \ov{\TS}
  = \ov{\T}. \label{eq:P_T-two}
\end{align}
The proof is an easy calculation from the definition. 

Let $\T = (\lambda_0\to \lambda_1 \to \cdots \to
\lambda_n)\in\Tab(n)$. In the next result, $\T[k]$ denotes the subpath
up to vertex $\lambda_k$ of the path $\T$, so, e.g., $\T[n-1] =
\ov{\T}$.

\begin{thm}\label{thm:idem-T-recursion}
Assume that $(J_n \mid n \in \N)$ is a separating sequence. Then, for any $\T
\in \Tab(n)$, $\idem_\T = \idem_{\ov{\T}} \, P_\T(J_n)$. Hence, 
$
	\idem_\T =\prod_{k=1}^{n} P_{\T[k]}(J_{k}) \,.
$
\end{thm}

\begin{proof}
We prove the first equality, as the second equality follows
immediately from the first by induction on $n$.  We claim that
$\idem_{\ov{\T}} \, P_\T(J_n)$ acts the same as $\idem_\T$ on all
basis elements $\{ \idem_\TS \mid \TS \in \Tab(n) \}$ of $\X_n$. That
is,
\[
  \idem_{\ov{\T}} \, P_\T(J_n) \cdot \idem_\TS = \delta_{\T,\TS}\, \idem_\TS
\]
for any $\TS \in \Tab(n)$. There
are three cases to the claim. First, if $\TS = \T$, then by
\eqref{eq:P_T-one} we have
\[
  \idem_{\ov{\T}} \, P_\T(J_n) \cdot \idem_\T = \idem_{\ov{\T}} \,
  \idem_\T = \idem_{\ov{\T}} \, \idem_{\ov{\T}} \, \idem(\lambda) =
  \idem_\T,
\]
which proves the claim in case $\TS=\T$. Next, if $\TS \ne \T$ but
$\ov{\TS} = \ov{\T}$ then the claim is immediate from
\eqref{eq:P_T-two}. So only the case $\TS \ne \T$ and $\ov{\TS} \ne
\ov{\T}$ remains. In this case we note that $\X_{n-1}\subset \X_n$,
and so $\idem_{\ov{\T}}$ and $J_n$ commute. Then
\[
  \idem_{\ov{\T}} \, P_\T(J_n) \cdot \idem_\TS = P_\T(J_n) \,
  \idem_{\ov{\T}} \, \idem_{\ov{\TS}}\, \idem(\mu) = 0,
\]
where $\TS \mapsto \mu$. This completes the proof of the claim. The
recursion formula now follows, since the claim implies that
\[
  \idem_{\ov{\T}} \, P_\T(J_n) = \idem_{\ov{\T}} \, P_\T(J_n) \cdot 1
  = \textstyle \sum_{\TS \in \Tab(n)} \idem_{\ov{\T}} \, P_\T(J_n)
  \cdot \idem_\TS = \idem_\T,
\]
as required. 
\end{proof}

We now return to JM-sequences. Our main result (Theorem
\ref{thm:central-P-recursion}) is a recursive formula for the
primitive central idempotents analogous to Theorem
\ref{thm:idem-T-recursion}. The following lemma holds the key
ingredients. If $\T \in \Tab(\lambda)$ we say that $\T$ has type
$\lambda$ and write $\type(\T) = \lambda$.

\begin{lem}
  Assume that $(J_n \mid n \in \N)$ is a JM-sequence in a
  multiplicity-free family $\{\A_n \mid n \ge 0 \}$.  Given any $\T
  \in \Tab(n)$, the content $c_\T(n)$ depends only on $\type(\T)$ and
  $\type(\ov{\T})$. 
  In particular, the polynomial $P_\T(J_n)$ in
Definition \ref{def:P_T} depends only on $\type(\T)$,
$\type(\ov{\T})$.
\end{lem}

\begin{proof}
The first statement follows immediately from Proposition
\ref{prop:central-J_n-property}.  The second is immediate from the
first and the definition of the $P_\T$.
\end{proof}

Hence, the following notation is well-defined. 

\begin{defn}\label{def:P-lambda-mu}
Suppose $(J_n \mid n \in \N)$ is a JM-sequence and $\T \in
\Tab(n)$. If $\lambda = \type(\T)$ and $\mu = \type(\ov{\T})$ then we
write $P^\lambda_\mu(J_n) = P_\T(J_n)$.
\end{defn}

We now arrive at the promised recursive description of the central
idempotents $\idem(\lambda)$. 

\begin{thm}\label{thm:central-P-recursion}
Assume that $(J_n \mid n \in \N)$ is a JM-sequence in a multiplicity-free
family. For any $\lambda \in \Irr(n)$, we have
\[
\idem(\lambda) \, = \, \sum_{\mu} \,P^\lambda_\mu(J_n)\cdot \idem(\mu),
\]
where $\mu$ varies over the set of immediate predecessors of $\lambda$
in the branching graph $\BG$.  
\end{thm}

\begin{proof}
We have $\idem(\lambda) = \sum_{\type(\T) = \lambda} \idem_\T$. By
Theorem \ref{thm:idem-T-recursion} and the above lemma, we
have
\[
  \idem(\lambda) = \sum_{\T\colon \type(\T) = \lambda} \idem_{\ov{\T}} \,
  P_\T(J_n) = \sum_{\mu} \bigg(
  \sum_{\substack{\T \colon \type(\T) = \lambda, \\ \type(\ov{\T}) = \mu}}
  \idem_{\ov{\T}} \bigg) \, P^\lambda_\mu(J_n). 
\]
To complete the proof, it suffices to show that 
\[
\sum_{\substack{\T\colon \type(\T)=\lambda, \\
  \type(\ov{\T}) = \mu}} \idem_{\ov{\T}} = \idem(\mu). 
\]
This conclusion is justified since any path $\ov{\T} \in \Tab(n{-}1)$ of
type $\mu$ extends uniquely to a path $\T \in \Tab(n)$ of type
$\lambda$ by the branching rule \ref{def:mfa}(c). Thus, the sum on the
left hand side above is a complete sum over all paths in $\Tab(n{-}1)$
of type $\mu$. The result follows.
\end{proof}

\section{Application: symmetric group algebras}\label{sec:symm-gp}
\noindent
Let $\Sym_n$ be the symmetric group on $n$ letters and $\Bbbk$ a field
of characteristic zero.  It is well known that the family
$\{\Bbbk \Sym_n \mid n \ge 0 \}$ is multiplicity-free (we take
$\Bbbk \Sym_0 = \Bbbk$); see \cite{VO}*{Theorem~2.1} for a proof of
this fact from first principles. Indeed, this multiplicity-free family
is the motivating example for our paper.

Vershik and Okounkov \cite{VO} give a complete and compelling account
of the representation theory of symmetric groups from the
multiplicity-free inductive viewpoint. In particular, they
\begin{itemize}
\item Compute the spectrum of the Young--Jucys--Murphy generators.
\item Show that the set of standard tableaux with $n$ boxes is in
  bijection with the set of all paths in the branching graph of length
  $n$; this also proves the branching rule.
\item Construct Young's seminormal and (when $\Bbbk = \C$) orthogonal
  forms for the irreducible representations.
\item Compute the irreducible characters (Murnaghan--Nakayama rule).
\end{itemize}
We cannot improve upon their story. But our story is about idempotents
in multiplicity--free families, so we are content to explain just
enough representation theory to be able to compute the
canonical idempotents 
\[
  \{ \idem_\T \mid \T \text{ a standard tableau with $n$ boxes}\}
\]
constructed in Definition \ref{def:idem-T}. We also
show that these idempotents coincide with the classical seminormal
idempotents constructed by Young.

We work with right modules in this section, in deference to
Schur--Weyl duality, discussed in the first paragraph of Section
\ref{sec:Brauer}.  Writing $i^\sigma$ for the image of $i$ under a
permutation $\sigma$, we define the product $\sigma\tau$ of two
permutations by $i^{\sigma\tau} = (i^\sigma)^\tau$, in order that
products of permutations agree with products of their Brauer diagrams.
We take for granted that the partitions of $n$ index the isomorphism
classes of irreducible representations and the set of standard
tableaux with $n$ boxes is in bijection with the set of all paths in
the branching graph of length $n$.  Under this bijection, the path
$\ov{\T}$ as defined in Section \ref{sec:mfa} corresponds to the
standard tableau $\ov{\T}$ obtained from the standard tableau $\T$ by
discarding the box containing the number $n$.  The labeled branching
graph for this family is depicted in Figure \ref{Bratteli-Sym}.

\begin{figure}[ht]
\begin{center}
\begin{tikzpicture}[xscale=3*\UNIT, yscale=-4.75*\UNIT]
	\coordinate (0) at (0,.6);
	\coordinate (11) at (0,1.4);
	\coordinate (21) at (-.90,2.1);\coordinate (22) at (.90,2.1);
	\foreach \x in {1, ..., 4}{\coordinate (3\x) at (-4+2*\x,3);}
	\foreach \x in {1, ..., 8}{\coordinate (4\x) at (-6+2*\x,4.25);}
	\foreach \x in {1, ..., 11}{\coordinate (5\x) at (-8.55+2.2*\x,5.6);}
\draw (0)--node[right, yshift=0ex, color=myred]{\tiny$0$} (11);
\draw (11)--node[right, yshift=-0.32ex, color=myred]{\tiny$1$} (21) (11)--node[right, pos=.32, color=myred]{\tiny$-1$} (22);
\draw (21)--node[right, pos=.57, color=myred]{\tiny$2$} (31) (21)--node[right, pos=.32, color=myred]{\tiny$-1$} (32);
\draw (22)--node[right, pos=.57, color=myred]{\tiny$1$} (32) (22)--node[right, pos=.32, color=myred]{\tiny$-2$} (33);
\draw (31)--node[right, pos=.57, xshift=.2ex, color=myred]{\tiny$3$} (41) (31)--node[right, pos=.42, color=myred]{\tiny$-1$} (42);
\draw (32)--node[right, pos=.52, xshift=.2ex, color=myred]{\tiny$2$} (42) (32)--node[right, pos=.37, color=myred]{\tiny$0$} (43) (32)--node[right, pos=.42, color=myred]{\tiny$-2$} (44);
\draw (33)--node[right, pos=.52, color=myred]{\tiny$1$} (44) (33)--node[right, pos=.52, color=myred]{\tiny$-3$} (45);
\draw (41)--node[right, pos=.40, color=myred]{\tiny$4$} (51) (41)--node[right, pos=.40, color=myred]{\tiny$-1$} (52);
\draw (42)--node[right, pos=.58, color=myred]{\tiny$3$} (52) (42)--node[right, pos=.58, color=myred]{\tiny$0$} (53) (42)--node[right, pos=.58, color=myred]{\tiny$-2$} (54);
%
\draw (43)--node[right, pos=.28, color=myred]{\tiny$2$} (53) (43)--node[right, pos=.25, color=myred]{\tiny$-2$} (55);
%
\draw (44)--node[right, pos=.68, color=myred]{\tiny$2$} (54) (44)--node[right, pos=.55, color=myred]{\tiny$0$} (55) (44)--node[right, pos=.48, color=myred]{\tiny$-3$} (56);
\draw (45)--node[right, pos=.40, color=myred]{\tiny$1$} (56) (45)--node[right, pos=.40, color=myred]{\tiny$-4$} (57);
\begin{scope}[every node/.style={fill=white}]
	\node at (0) {\footnotesize$\bb$};
	\node at (11) {\tPART{1}};

	\node at (22) {\tPART{1,1}};
	\node at (21) {\tPART{2}}; 

	\node at (33) {\tPART{1,1,1}};
	\node at (32) {\tPART{2,1}};
	\node at (31) {\tPART{3}};

	\node at (45) {\tPART{1,1,1,1}};
	\node at (44) {\tPART{2,1,1}};
	\node at (43) {\tPART{2,2}};
	\node at (42) {\tPART{3,1}};
	\node at (41) {\tPART{4}};

	\node at (57) {\tPART{1,1,1,1,1}};
	\node at (56) {\tPART{2,1,1,1}};
	\node at (55) {\tPART{2,2,1}};
	\node at (54) {\tPART{3,1,1}};
	\node at (53) {\tPART{3,2}};
	\node at (52) {\tPART{4,1}};
	\node at (51) {\tPART{5}};
\end{scope}
\foreach \x in {-5.1,0,5.1} {\node at (\x , 6.2) {$\vdots$}; }
\end{tikzpicture}
\end{center}
\caption{Branching graph for the multiplicity-free family $\{\Bbbk \Sym_n\}$. 
The edge labels are computed in Proposition \ref{prop:c_T-symm}.}
\label{Bratteli-Sym}
\end{figure}
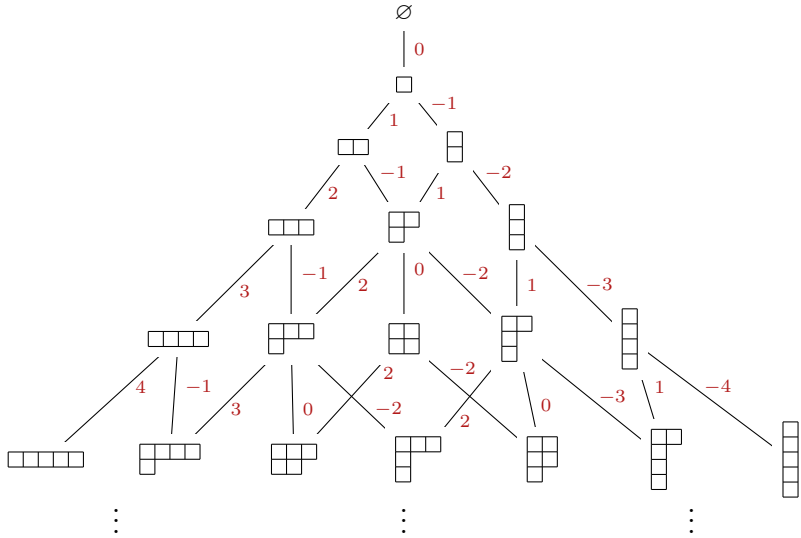

Young's construction of the irreducible representations is in terms of
the so-called Young symmetrizers. Let us recall the definitions; see
e.g., \cite{Fulton}.

\begin{defn}
  Given a tableau $\T$ of $n$ boxes, let $\mathsf{R}(\T)$ be the
  subgroup of $\Sym_n$ consisting of all $w$ which stabilize the rows
  of $\T$; similarly, let $\mathsf{C}(\T)$ be the subgroup of $\Sym_n$
  consisting of all $w$ which stabilize the columns of $\T$. Put
  \[
    \fraka_\T = \textstyle \sum_{w \in \mathsf{R}(\T)} w, \qquad
    \frakb_\T = \textstyle \sum_{w \in \mathsf{C}(\T)} \sgn(w)\,w .
  \]
  The sums $\fraka_\T, \frakb_\T$ and their products
  $\fraka_\T\frakb_\T$, $\frakb_\T\fraka_\T$ taken in either order
  are called Young symmetrizers. We put 
  \[
   \yy_\T = \frac{1}{h(\lambda)} \fraka_\T\frakb_\T =
   \frac{1}{h(\lambda)} \sum_{v \in \mathsf{R}(\T)} v \sum_{w \in
     \mathsf{C}(\T)} \sgn(w)\,w,
  \]
  where $h(\lambda) = \frac{n!}{n(\lambda)}$ and $n(\lambda)$ is the
  number of standard tableaux of shape $\lambda$. (Then $h(\lambda)$
  is equal to the product of all the hook lengths in $\T$; this
  depends only on the shape $\lambda$ of $\T$.)
\end{defn}

If $w \in \Sym_n$ and $\T$ is a tableau with $n$ boxes then $w \cdot
\T$ is the tableau obtained by replacing each number $i \in \T$ by
$w(i)$, for $i = 1, \dots, n$. We have $\mathsf{R}(w \cdot \T) = w
\mathsf{R}(\T) w^{-1}$ and $\mathsf{C}(w \cdot \T) = w \mathsf{C}(\T)
w^{-1}$. Thus
\begin{equation}\label{eq:wT-action}
  \fraka_{w \cdot \T} = w \fraka_\T w^{-1}, \quad \frakb_{w \cdot \T} =
  w \frakb_\T w^{-1}
\end{equation}
for any $w \in \Sym_n$. 
The $\yy_\T$ are idempotents in $\Bbbk \Sym_n$; these idempotents
are sometimes called \emph{Young's idempotents}. The right ideal
\begin{equation}
  S^\lambda = \yy_\T \Bbbk \Sym_n   
\end{equation}
is an irreducible $\Bbbk \Sym_n$-module, where $\T$ is any standard
tableau of shape $\lambda$.  It is known that $\yy_\T \Bbbk \Sym_n
\cong \yy_{\T'} \Bbbk \Sym_n$ if and only if $\T, \T'$ have the same
shape, so the isomorphism type of the right ideal $\yy_\T \Bbbk
\Sym_n$ depends only on the shape of $\T$. It is also well known (see
e.g., \cite{Curtis-Reiner}*{\S28}) that
\begin{equation}\label{eq:y_Tdecomp}
  \Bbbk \Sym_n \ = \bigoplus_{\T} \yy_\T \Bbbk \Sym_n,
\end{equation}
where the sum is taken over the set of all standard tableaux $\T$ of
$n$ boxes. This is a decomposition as a direct sum of simple right
ideals, but unfortunately the family $\{\yy_\T\}$ of primitive
idempotents is \emph{not} pairwise orthogonal,\footnote{\,Not quite, but
  almost! It can be shown (see e.g., \cite{Stembridge}*{Prop.~1}) that
  for each pair $\T, \T'$ of standard tableaux of $n$ boxes, at least
  one of the products $\yy_\T \cdot \yy_{\T'}$, $\yy_{\T'} \cdot
  \yy_\T$ must be zero.} as already noted by Young; see
\cite{Stembridge} for an explicit counterexample.

We put $z_n$ equal to the formal sum of all transpositions in
$\Sym_n$, regarded as an element of $\Bbbk \Sym_n$.  This conjugacy
class sum is an element of the center of $\Bbbk \Sym_n$ for each
$n$. We wish to show that the elements
\begin{equation}
  J_n = z_n - z_{n-1} = (1,n) + (2,n) + \cdots + (n-1,n)
\end{equation}
(written in the cycle notation for permutations) define a JM-sequence
in the sense of Definition \ref{def:JM-defn}. 

\begin{prop}\label{prop:z_n-Sym}
  Let $\lambda$ be a partition of $n$ and $\T$ a standard tableau of
  shape $\lambda$. Then the central element $z_n$ acts by right
  multiplication on $y_\T \Bbbk \Sym_n$ as the scalar
  $a_\lambda = \xi(\lambda) - \xi(\lambda')$, where $\lambda'$ is the
  transpose of $\lambda$ and
  $\xi(\lambda) = \sum \binom{\lambda_m}{2}$, summed over the parts
  $\lambda = (\lambda_1, \lambda_2, \dots)$ of $\lambda$.
\end{prop}

\begin{proof}
  Since $z_n$ is in the center of $\Bbbk \Sym_n$, we know there is a
  scalar $a_\lambda \in \Bbbk$ with $v \cdot z_n = a_\lambda\, v$, for
  all $v \in \yy_\T \Bbbk \Sym_n$.  In particular,
  $y_\T \cdot z_n = a_\lambda \, y_\T$. By definition of $y_\T$, this
  equality becomes
  \[
    \sum_{\alpha \in \mathsf{R}(\T)} \sum_{\beta \in \mathsf{C}(\T)}
    \sgn(\beta)\, \alpha \beta \ \cdot\! \sum_{1 \le i < j \le n} (i,j) 
    \ = \   
    a_\lambda \, \sum_{\alpha \in \mathsf{R}(\T)} \sum_{\beta \in
      \mathsf{C}(\T)} \sgn(\beta)\, \alpha \beta,
  \]
  where $(i,j)$ denotes the transposition interchanging $i$ and 
  $j$. To compute $a_\lambda$ we compare coefficients of the identity
  permutation on both sides of the equation, which gives
  \[
    \sum_{\alpha \in \mathsf{R}(\T)} \sum_{\beta \in \mathsf{C}(\T)}
    \sum_{1 \le i < j \le n}  \sgn(\beta)\, \delta_{\alpha\beta,
      (i,j)} = a_\lambda,
  \]
  where $\delta_{\sigma,\tau} = 1$ if $\sigma = \tau$ and $0$
  otherwise (for $\sigma, \tau \in \Sym_n$). It is easy to see that
  $\delta_{\alpha\beta, (i,j)} = 0$ unless $i$ and $j$ lie in the same
  row or column of $\T$, since otherwise the product $\alpha\beta$
  must change more than just $i$ and $j$. So we are reduced to
  counting solutions of the equation $\alpha\beta = (i,j)$ of the form
  $\alpha = (i,j)$ where $i,j$ lie in the same row of $\T$ or
  $\beta = (i,j)$ where $i,j$ lie in the same column of $\T$.  In
  other words, we need to count the number of pairs $(i,j)$ with $i<j$
  in a row of $\T$, and, with opposite sign, the number of pairs
  $(i,j)$ with $i<j$ in a column of $\T$.  This gives the desired
  result $a_\lambda = \xi(\lambda) - \xi(\lambda')$.
\end{proof}

Here is a combinatorial procedure for computing the statistic
$\xi(\lambda)$ for a given shape $\lambda$. Insert the numbers
$0, 1, 2, \dots, \lambda_m-1$ in order into the $m$th row of the
diagram of shape $\lambda$, for each $m$. Then clearly $\xi(\lambda)$
is equal to the sum of the numbers in the boxes. Note that the
insertion process just described is equivalent to inserting a $j-1$ in
each box of the $j$th column of the diagram. So $\xi(\lambda)$ is the
sum of all the numbers in this numbering.

On the other hand, if we insert $i-1$ in each box of the $i$th row of the
diagram of shape $\lambda$, then $\xi(\lambda')$ is the sum of all the
numbers in this numbering, where $\lambda'$ is the transpose of $\lambda$. 

This implies that if we attach the statistic $(j-1) - (i-1) = j-i$ to
the box in row $i$ and column $j$ in $\T$ (this statistic is called
the \emph{content} of the box) then the sum of all the statistics is
$a_\lambda = \xi(\lambda) - \xi(\lambda')$. Thus, we see that
\begin{equation}\label{eq:a_lambda-sum}
  a_\lambda = \textstyle \sum_{(i,j)} j-i\,,
\end{equation}
where the sum is taken over the positions $(i,j)$ indexing all the
boxes in the diagram of shape $\lambda$. This interpretation of
$a_\lambda$ will be used to prove the following result.

\begin{prop}\label{prop:c_T-symm}
  Suppose that $\T\in\Tab(n)$ has shape $\lambda$. For any
  $1\leq k\leq n$, the eigenvalue $c_\T(k)$ of the action of $J_k$ on
  the Gelfand--Tsetlin basis element $v_\T$ indexed by $\T$ is the
  content $j-i$, where the box containing $k$ is located in row $i$
  and column $j$ in the tableau $\T$.
\end{prop}

\begin{proof}
We proceed by induction on $n$. For $n=1$ the result is clear:
$c_\T(1) = 0$ as $J_1 = 0$. Let $n > 1$ and let $\T \in \Tab(n)$ be a
standard tableau of shape $\lambda$, some partition of $n$. By the
inductive hypothesis, $c_{\ov{\T}}(k)$ has the desired value for any
$k \le n-1$. By Proposition \ref{prop:GJM-separation}(a), $c_\T(k) =
c_{\ov{\T}}(k)$ for all $k<n$, so $c_\T(k)$ has the desired value for
all $k<n$. Thus, it suffices to compute the value $c_\T(n)$.

By Proposition \ref{prop:central-J_n-property} we have $c_\T(n) =
a_\lambda - a_\mu$, where $\ov{\T} \mapsto \mu$. By equation
\eqref{eq:a_lambda-sum}, it follows that $c_\T(n) = a_\lambda - a_\mu
= j-i$, where the box in $\T$ containing $n$ occurs in position
$(i,j)$. The result is proved.
\end{proof}

\begin{rmk}\label{rmk:addable}
If we record the statistic $j-i$ in each box $(i,j)$ of the Young
diagram of shape $\lambda$, then the resulting tableau is constant
along diagonals. Recall that a box in a Young diagram of shape
$\lambda$ is \emph{removable} if excising it results in another Young
diagram. Similarly, a box not in the shape $\lambda$ is \emph{addable}
if including it results in a Young diagram. Since removable boxes are
always the last box in their row or column, it is clear that no two
removable boxes in $\lambda$ can lie on the same diagonal. The same
conclusion applies to addable boxes. Hence, no two removable (or
addable) boxes for a shape $\lambda$ can have the same content. This is
needed in the proof of Corollaries \ref{cor:Sym_n-separation} and
\ref{cor:Brauer-JM-sep}.
\end{rmk}

\begin{cor}\label{cor:Sym_n-separation}
  The sequence $(J_k \mid k \in \N)$ is a JM-sequence in the sense of
  Definition \ref{def:JM-defn}.
\end{cor}

\begin{proof}
  Since $J_k = z_k - z_{k-1}$ and $z_k \in Z(\Bbbk \Sym_k)$ for all
  $1\leq k\leq n$, it follows that each $J_k\in\X_n$ and that
  $(J_k)_{k\in\N}$ is additively central. We use Proposition
  \ref{prop:separation-generation-equivalent} to verify that it is
  also a separating sequence.  Proposition \ref{prop:c_T-symm}
  computes the content vectors $c_\T = (c_\T(1), \dots, c_\T(n))$ for
  each $\T \in \Tab(n)$. 

  Let $\T[k]$ denote the standard tableau obtained from $\T \in
  \Tab(n)$ by removing all boxes containing numbers larger than
  $k$. Assume that $\TS\neq\T$.  We show $c_\TS\neq c_\T$. Find the
  smallest $k \leq n$ at which the tableaux $\TS, \T$ differ. That is,
  $\TS[k-1] = \T[k-1]$, yet $\TS[k] \ne \T[k]$.  By Remark
  \ref{rmk:addable}, the contents of the addable boxes yielding
  $\TS[k]$ and $\T[k]$ differ.  Appealing to Proposition
  \ref{prop:c_T-symm}, we conclude that $c_\TS(k) \neq c_\T(k)$. This
  completes the proof.
\end{proof}

For the sake of completeness, we give another formula for the central idempotent
$\idem(\lambda)$ in terms of Young symmetrizers; it was obtained by Young 
in his first two papers, published
in 1900 and 1901. (See \cite{Curtis}*{Ch.~II, \S5} for an historical
account of these developments.)

\begin{prop}[Young]
  For each $\lambda \vdash n$, $\idem(\lambda) = \frac{1}{h(\lambda)}
  \sum_\T \yy_\T$, where the sum is taken over all 
  tableaux $\T$ (not necessarily standard) of shape $\lambda$.
\end{prop}

The proof is an easy exercise, cf.\ \cite{Simon}*{Cor.~VI.3.7}.

Recall that Young \cite{Young} found a family of primitive idempotents
$\{e_\T\}$, also indexed by the set of standard tableaux of $n$ boxes,
which are pairwise orthogonal and sum to $1$. These idempotents are
part of Young's seminormal form, so we call them Young's 
\demph{seminormal idempotents}. R.~M.~Thrall \cite{Thrall} (see also
\cite{Garsia}*{2.16}, \cite{Lascoux}) found the following recursive
description of the $e_\T$. For each standard tableau $\T$ of $n$
boxes, the element $e_\T$ of $\Bbbk \Sym_n$ may be defined by
\begin{equation}\label{eq:e_T-def}
  e_\T = 
  \begin{cases}
    e_{\ov{\T}} \cdot \yy_\T \cdot e_{\ov{\T}}  & \text{ if } n > 1, \\
    1 &  \text{ if } n = 1,
  \end{cases}
\end{equation}
where $\ov{\T}$ is the standard tableau obtained from $\T$ by removing
the box containing $n$.

So we now have two families $\{e_\T\}$, $\{\idem_\T\}$ of pairwise
orthogonal primitive idempotents, both indexed by the set of standard
tableaux of $n$ boxes. One might ask how the
two families are related. Here is the answer.

\begin{prop}\label{thm:Giaquinto=Thrall}
  For any standard tableau $\T$ of $n$ boxes, we have $\idem_\T =
  e_\T$. So the canonical idempotents of Definition \ref{def:idem-T}
  are Young's seminormal idempotents in the case of symmetric group
  algebras.
\end{prop}

\begin{proof} 
  This follows immediately from Corollary
  \ref{cor:idem-characterization} once we observe that
  $e_{\ov{\T}} e_\T = e_\T$ for all standard tableaux $\T$. Indeed,
  this relation is clear from Thrall's recursive definition of the
  $e_\T$; see \eqref{eq:e_T-def}.
\end{proof}

\begin{rmk}
Thrall's recursive description of the seminormal idempotents depends
on the Young symmetrizers, while the simpler recursion obtained by the
methods of this paper does not.
\end{rmk}

\begin{examples}\label{exam:symm-idems}
We compute a number of $\idem(\lambda)$ recursively using Theorem
\ref{thm:central-P-recursion} and Proposition \ref{prop:c_T-symm},
referring to the branching graph in Figure \ref{Bratteli-Sym}. Of course
$\idem(\ttPART{1}) = 1$. 


\smallskip \noindent\emph{Primitive central idempotents for $n=2$:}
\begin{align*}
\idem(\ttPART{2}) &= P^{\,\ttPART{2}}_{\ttPART{1}}\,\idem(\ttPART{1}) = P^{\,\ttPART{2}}_{\ttPART{1}} = {\tfrac12(J_2+1)} \\
\idem(\ttPART{1,1}) &= P^{\,\ttPART{1,1}}_{\ttPART{1}} \,\idem(\ttPART{1}) = P^{\,\ttPART{1,1}}_{\ttPART{1}}  = -{\tfrac12(J_2 - 1)}. 
\\[1ex]
\intertext{\emph{Primitive central idempotents for $n=3$:}}
\idem(\ttPART{3}) &= P^{\,\ttPART{3}}_{\ttPART{2}} \,\idem(\ttPART{2}) 
	= {\tfrac13(J_3+1)}\,\idem(\ttPART{2}) \\[1.5ex]
\idem(\ttPART{2,1}) &= P^{\,\ttPART{2,1}}_{\ttPART{2}} \,\idem(\ttPART{2}) + P^{\,\ttPART{2,1}}_{\ttPART{1,1}} \,\idem(\ttPART{1,1}) 
	= -{\tfrac13(J_3-2)}\, \idem(\ttPART{2}) + {\tfrac13(J_3+2)}\, \idem(\ttPART{1,1}) \\[1.5ex]
\idem({\!}_{\!}\raisebox{-.5ex}{\ttPART{1,1,1}}{\!}_{\!}) &= P^{\,\ttPART{1,1,1}}_{\ttPART{1,1}} \,\idem(\ttPART{1,1}) 
	=  - {\tfrac13(J_3-1)}\, \idem(\ttPART{1,1}).
\\[1ex]
\intertext{\emph{Primitive central idempotents for $n=4$:}}
\idem(\ttPART{4}) &= 
	P^{\,\ttPART{4}}_{\ttPART{3}} \,\idem(\ttPART{3}) 
	= {\tfrac14(J_4+1)}\,\idem(\ttPART{3})\\[1.5ex]
\idem(\ttPART{3,1}) &= 
	P^{\,\ttPART{3,1}}_{\ttPART{3}} \,\idem(\ttPART{3})  
	+ P^{\,\ttPART{3,1}}_{\ttPART{2,1}} \,\idem(\ttPART{2,1}) 
	= -{\tfrac14(J_4-3)}\,\idem(\ttPART{3}) 
	+ {\tfrac18(J_4+2)J_4}\,\idem(\ttPART{2,1}) \\[1.5ex]
\idem(\ttPART{2,2}) &= 
	P^{\,\ttPART{2,2}}_{\ttPART{2,1}} \,\idem(\ttPART{2,1}) 
= -{\tfrac14(J_4+2)(J_4-2)}\,\idem(\ttPART{2,1}) \\[1.5ex]
\idem({\!}_{\!}\raisebox{-.5ex}{\ttPART{2,1,1}}{\!}_{\!}) &= 
	P^{\,\ttPART{2,1,1}}_{\ttPART{2,1}} \,\idem(\ttPART{2,1}) 
	+ P^{\,\ttPART{2,1,1}}_{\ttPART{1,1,1}} \,\idem({\!}_{\!}\raisebox{-.5ex}{\ttPART{1,1,1}}{\!}_{\!}) 
	= {\tfrac18(J_4-2)J_4}\,\idem(\ttPART{2,1}) 
	+ {\tfrac14(J_4+3)}\,\idem({\!}_{\!}\raisebox{-.5ex}{\ttPART{1,1,1}}{\!}_{\!}) \\[.5ex]
\idem({\!}_{\!}\raisebox{-.75ex}{\ttPART{1,1,1,1}}{\!}_{\!}) &= P^{\,\ttPART{1,1,1,1}}_{\ttPART{1,1,1}} \,\idem({\!}_{\!}\raisebox{-.5ex}{\ttPART{1,1,1}}{\!}_{\!}) 
	= -{\tfrac14(J_4-1)} \,\idem({\!}_{\!}\raisebox{-.5ex}{\ttPART{1,1,1}}{\!}_{\!}).
\end{align*}
We note that the summands in each $\idem(\lambda)$ are the
various $\idem_\T$ in that block, so the $\idem_\T$ are recoverable
from the above expressions.
\end{examples}

\section{Application: Brauer algebras}\label{sec:Brauer}\noindent
In \cite{Brauer}, Brauer defined a finite dimensional algebra
$\B_n(m)$ over $\C$ in order to quantify the invariants of orthogonal
groups. If $E$ is an $m$-dimensional vector space over $\C$ then
$\GL(E) \cong \GL_m(\C)$ acts naturally (on the left) on $E$, this
action extends diagonally to one on $E^{\otimes n}$. The group
$\Sym_n$ acts by place-permutation (on the right) on $E^{\otimes
  n}$. These actions commute, so by linearly extending the actions to
representations, the tensor space $E^{\otimes n}$ is a
$(\C\GL(E),\C\Sym_n)$-bimodule. Classical Schur--Weyl duality
\cite{Schur-27} says that the image of each representation in
$\End_\C(E^{\otimes n})$ is equal to the full centralizer of the
other. This duality elegantly expresses the fundamental duality
between the representation theories of general linear groups and
symmetric groups.

Brauer extended the action of the symmetric group algebra to one of
the algebra $\B_n(m)$ such that when the left action of $\GL(E)$ is
restricted to the orthogonal group $\Orth(E)$, Schur--Weyl duality
also holds for the resulting $(\C\Orth(E), \B_n(m))$-bimodule
structure on $E^{\otimes n}$. This duality relates the representation
theory of orthogonal groups and Brauer algebras. Brauer algebras also
have connections to low-dimensional topology and knot theory; see
e.g.\ \cites{Kauffman,Birman-Wenzl,Fishel-Grojnowski}.

Let $k,l$ be positive integers of the same parity, so that $k+l$ is
even. A Brauer $(k,l)$-diagram is an undirected graph with $k+l$
vertices, such that each vertex is an endpoint of exactly one
edge. Conventionally, the vertices are arranged in two rows within a
rectangle, with $k$ vertices (the \emph{top} vertices) along the top
boundary and $l$ vertices (the \emph{bottom} vertices) along the
bottom boundary, with the edges drawn in the interior of the rectangle
in such a way that intersecting edges cross transversally. For
example, the graph
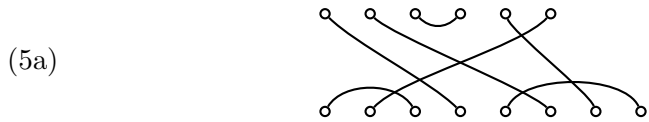
\begin{gather}\label{pic:Brauer-dia}
\begin{tikzpicture}[scale = 0.4,thick, baseline={(0,-2ex)}] 
\tikzstyle{vertex} = [shape = circle, minimum size = 3.5pt, inner sep =
1pt] 
\node[vertex] (G--8) at (10.5, -2.25) [shape = circle, draw] {}; 
\node[vertex] (G--5) at (6.0, -2.25) [shape = circle, draw] {}; 
\node[vertex] (G--7) at (9.0, -2.25) [shape = circle, draw] {}; 
\node[vertex] (G-6) at (7.5, 1) [shape = circle, draw] {}; 
\node[vertex] (G--6) at (7.5, -2.25) [shape = circle, draw] {}; 
\node[vertex] (G-1) at (0.0, 1) [shape = circle, draw] {}; 
\node[vertex] (G--4) at (4.5, -2.25) [shape = circle, draw] {}; 
\node[vertex] (G--3) at (3.0, -2.25) [shape = circle, draw] {}; 
\node[vertex] (G--1) at (0.0, -2.25) [shape = circle, draw] {}; 
\node[vertex] (G--2) at (1.5, -2.25) [shape = circle, draw] {}; 
\node[vertex] (G-2) at (1.5, 1) [shape = circle, draw] {}; 
\node[vertex] (G-5) at (6.0, 1) [shape = circle, draw] {}; 
\node[vertex] (G-3) at (3.0, 1) [shape = circle, draw] {}; 
\node[vertex] (G-4) at (4.5, 1) [shape = circle, draw] {}; 
\draw (G--8) .. controls +(-0.6, 1.25) and +(0.6, 1.25) .. (G--5); 
	\draw (G-6) .. controls +(-1, -1.1) and +(1, 1.1) .. (G--2); 
\draw (G-1) .. controls +(1, -1) and +(-1, 1) .. (G--4); 
\draw (G--3) .. controls +(-0.6, 1) and +(0.6, 1) .. (G--1); 
\draw (G-3) .. controls +(0.5, -0.5) and +(-0.5, -0.5) .. (G-4); 
	\draw (G-2) .. controls +(0.75, -0.75) and +(-0.75, 0.75) .. (G--6); 
	\draw (G-5) .. controls +(0.25, -0.65) and +(-0.25, 0.65) .. (G--7); 
\end{tikzpicture} 
\end{gather}
is a Brauer $(6,8)$-diagram.  Edges connecting two vertices in the
same row are called \demph{horizontal edges}.  All other edges must
have one top and one bottom endpoint, such edges are \demph{through
  edges}. The \emph{rank} of a diagram is the number of through edges.

Let $\Bbbk$ be a ring and $\parm \in \Bbbk$ a distinguished parameter.
Multiplication of Brauer diagrams is defined as follows. Given a
$(k,l)$-diagram $b$ and an $(l,m)$-diagram $b'$, place $b$ above $b'$
and identify the $i$th bottom vertex of $b$ with the $i$th top vertex
of $b'$. Let $N=N(b,b')$ be the number of interior loops in the new
graph and let $b''$ be that graph with its loops and intermediate
vertices omitted. Then $b''$ is a $(k,m)$-diagram, and we define
\begin{equation}\label{eq:dia-mult}
   b b' = \parm^N  (b \compose b'), \quad \text{where } b \compose
   b' = b''.
\end{equation} 
The $(k,m)$-diagram $b'' = b \compose b'$ is the \emph{composite
  diagram} of $b, b'$. Note that the parameter $\parm$ keeps track of
the number of discarded interior loops. In case $k=m$ we call the
diagram $b''$ simply an $m$-diagram.

The Brauer algebra over $\Bbbk$ with parameter $\parm$ is denoted by
$\B_n(\parm)$, and is defined to be the $\Bbbk$-span of the set of
$n$-diagrams.  Extended linearly, the multiplication rule $(b,b')
\mapsto bb'$ in \eqref{eq:dia-mult} defines an associative
multiplication on $\B_n(\parm)$. An \demph{identity edge} in an
$n$-diagram is an edge connecting the $i$th vertices in the top and
bottom rows; the $n$-diagram in which all edges are identity edges is
the unit element of $\B_n(\parm)$. Brauer $n$-diagrams in which every
edge is a through edge will be identified with permutations; note that
multiplication of Brauer diagrams coincides with multiplication of
permutations in case both diagrams are permutations, so $\Bbbk\Sym_n$
is a subalgebra of $\B_n(\parm)$. Clearly $\B_1(\parm) \cong \Bbbk$;
we agree to interpret $\B_0(\parm) = \Bbbk$.

Let $\B_{k,l}(\parm)$ be the $\Bbbk$-span of the set of
$(k,l)$-diagrams. Multiplication of Brauer diagrams makes this into a
$(\B_k(\parm), \B_l(\parm))$-bimodule with $\B_k(\parm)$ acting by
left multiplication and $\B_l(\parm)$ by right multiplication. This
bimodule structure will be used below to construct representations of
Brauer algebras.

Semisimplicity of $\B_n(\delta)$ over $\C$ was studied in \cite{Brown}
in the case when $\delta$ is a positive integer: he showed that
$\B_n(\delta)$ is semisimple if and only if $\delta \ge n-1$. Still
working over $\C$, Hanlon and Wales \cite{Hanlon-Wales:90} conjectured
that $\B_n(\delta)$ is always semisimple if $\delta \in \C$ is not an
integer; the conjecture was proved by Wenzl \cite{Wenzl}, who also
parametrized the simple modules and established the branching diagram.
Further work on semisimplicity of Brauer algebras, including
semisimplicity over other fields, can be found in \cites{Doran-et-al,
  Rui, Cox-et-al}.

We assume for the remainder of this section that $\Bbbk$ is a field of
characteristic zero and $\delta \in \Bbbk$ is not an integer. This
assumption ensures that $\B_n(\delta)$ is split semisimple over
$\Bbbk$. Under this assumption, we show that the Brauer algebras form
a multiplicity-free family, identify a JM-sequence for this family,
develop eigenvalue formulas, and compute central idempotents using
Theorem \ref{thm:central-P-recursion}.

In order to simplify the notation, we suppress the parameter $\parm$,
writing $\B_n = \B_n(\parm)$ from now on.  There is a natural unital
embedding
\[
  \iota: \B_n \into \B_{n+1} 
\]
given by sending an $n$-diagram to the corresponding $(n{+}1)$-diagram
obtained by appending an identity edge on the right (connecting two
additional vertices). We identify $\B_{n}$ as a unital subalgebra of
$\B_{n+1}$, for each $n$, without further mention of $\iota$.

We write $(i,j)$ for the $n$-diagram corresponding to a transposition
$(i,j) \in \Sym_n$; this is the diagram with through edges connecting
the $i$th and $j$th top vertices to the $j$th and $i$th bottom ones,
respectively, with all other edges identity edges. Similarly,
$\ov{(i,j)}$ is the $n$-diagram with horizontal edges connecting the
$i$th and $j$th vertices in each row, and all other edges identity
edges. We set
\begin{equation}\label{eq:s_i-e_i}
  s_i = (i,i+1); \quad e_i = \ov{(i,i+1)}, \qquad \text{ any } i < n.
\end{equation}
It is easy to see that $\B_n$ is generated by the $s_i, e_i$ for $1
\le i \le n-1$. Defining relations satisfied by these generators can
be found in \cite{Nazarov}. Note that $e_i^2 = \parm e_i$, so
$\parm^{-1} e_i$ is idempotent. Any $e_i$ generates the two-sided
ideal spanned by all diagrams with at least two horizontal edges; the
quotient by this ideal is isomorphic to $\Bbbk \Sym_n$.

Our next task is to construct the irreducible (right) $\B_n$-modules.
For this purpose it is useful to apply some general observations from
\cite{Green}*{\S6.2}. The applicability of these ideas to diagram
algebras was demonstrated in \cites{Martin-Saluer, Martin:96,
  Doran-et-al, Martin-Ryom-Hansen, Cox-et-al, CDM}; here we more or
less follow the summary outline at the beginning of \cite{Cox-et-al}.
In general, then, let $A$ be an algebra over a field $\Bbbk$ and $e
\in A$ an idempotent. The rule
\[
M \mapsto Me
\]
defines an exact functor $\F$ (often called the ``Schur functor'')
from right $A$-modules to right $eAe$-modules.  The functor $\F$ takes
irreducible modules to irreducible modules, or zero. More precisely,
we have the following result.

\begin{thm}[\cite{Green}*{(6.2g)}] 
\label{thm:Green}
  Let $\{L(\lambda): \lambda \in \Lambda \}$ be a full set of pairwise
  non-isomorphic irreducible right $A$-modules, and let \[ \Lambda^e =
  \{\lambda \in \Lambda: L(\lambda)e \ne 0\}. \] Then $\{L(\lambda)e:
  \lambda \in \Lambda^e \}$ is a full set of pairwise non-isomorphic
  irreducible right $eAe$-modules.
\end{thm}

Note that right $A$-modules annihilated by $e$ are equivalent to right
$A/AeA$-modules.
Thus, the irreducible right $A$-modules $L(\lambda)$ with $\lambda \in
\Lambda \backslash \Lambda^e$ are a full set of irreducible
$A/AeA$-modules. If $A$ is finite dimensional, this reduces the
problem of finding an indexing set $\Lambda$ for the irreducible
$A$-modules to the same problem for the smaller algebras $eAe$,
$A/AeA$.

There is another functor $\G$ going from right $eAe$-modules to right
$A$-modules, defined by $\G(N) = N \otimes_{eAe} eA$. This functor,
which was also considered in \cite{Green}*{\S6.2}, is a right inverse
to $\F$, i.e., $\F(\G(N)) \cong N$, so $\G$ is a full embedding.%
\footnote{\,In \cite{Cox-et-al}, the functors $\F,\G$ are called ``localization'' and
``globalization'' functors, respectively.} 
Furthermore,
\cite{Green}*{(6.2e)} shows that $\G(N)$ always has a unique maximal
proper submodule whenever $N$ is irreducible.

In case $A$ is semisimple, it follows that $\G$ must take irreducible
$eAe$-modules to irreducible $A$-modules (and the unique maximal
proper submodule is zero). Thus, for irreducible $A$-modules $M$ such
that $Me \ne 0$, we have $\G(\F(M)) \cong M$. So, $\G$ is also a left
inverse to $\F$. Thus, in the semisimple case, the functors $\F$ and
$\G$ implement an equivalence of categories between $A$-modules not
killed by $e$ and $eAe$-modules. Since by semisimplicity $A \cong AeA
\bigoplus A/AeA$, and the $A$-modules killed by $e$ are the
$A/AeA$-modules, it follows that the $A$-modules not killed by $e$ are
the same as the $AeA$-modules. To summarize:
\begin{prop}\label{prop:AeA=eAe}
If A is semisimple, then: 
\begin{enumerate}
\item $\G$ takes irreducible to irreducibles. 
\item The functors $\F$, $\G$ induce an equivalence of categories
  between $AeA$-modules and $eAe$-modules.
\end{enumerate}
\end{prop}

Now we apply the above observations to the algebra $\B_n$, taking $e$
to be the idempotent $\xi_n = \parm^{-1} e_{n-1}$, with $e_{n-1}$ as
in \eqref{eq:s_i-e_i}.  This immediately gives functors $\F_n,
\G_{n-2}$ as above, defined by the rules
\[
  \F_n(M) = M\xi_n, \qquad \G_{n-2}(N) = N \otimes_{\xi_n \B_n \xi_n}
  \xi_n \B_n
\]
for any right $\B_n$-module $M$, any right $\B_{n-2}$-module $N$. A
crucial fact about the idempotent $\xi_n$ is that there is an
isomorphism of algebras
\begin{equation}\label{eq:rec-iso}
  \B_{n-2} \cong \xi_n \B_n \xi_n
\end{equation}
for each $n \ge 2$. The isomorphism is given by the rule $b \mapsto
\xi_n b \xi_n$, for $b \in \B_{n-2}$; note that it maps the unit
element of $\B_{n-2}$ to $\xi_n$.  Furthermore, $\xi_n$ commutes
pointwise with $\B_{n-2}$:
\begin{equation}\label{eq:xi_n-commutes}
  \xi_n b = b \xi_n, \quad\text{for all } b \in \B_{n-2}.
\end{equation}
In consequence, we have $\B_{n-2} \cong \B_{n-2}\xi_n =
\xi_n\B_{n-2}\xi_n$.  If $\Irr(n)$ is a set indexing the irreducible
$\B_n$-modules and $\Irr^n$ a set indexing the irreducible $\B_n/\B_n
\xi_n \B_n$-modules, then it follows from \eqref{eq:rec-iso} and the
preceding remarks that
\[
\Irr(n) = \Irr^n \,\sqcup\,  \Irr(n-2) .
\]
Since $\B_n/\B_n \xi_n \B_n$ is isomorphic to $\Bbbk \Sym_n$, we can
set $\Irr^n = \{ \lambda \mid \lambda \vdash n \}$. It is trivial to
compute $\Lambda_0$ and $\Lambda_1$ (as $\B_0 \cong \B_1 \cong
\Bbbk$), so it immediately follows by induction on $n$ that
\begin{equation*}
\Irr(n) = \{ \lambda \mid \lambda \vdash n-2l \qand 0 \le 2l \le n \}.
\end{equation*}

Now that we know an indexing set for the irreducible $\B_n$-modules,
we turn to the problem of constructing them. We will follow the
approach of \cite{Doran-et-al}, using the $(\B_k, \B_n)$-bimodule
$\B_{k,n} = \B_{k,n}(\delta)$ discussed above, where $k \le n$ has the
same parity as $n$. Let $\B^0_{k,n}$ be the span of the
$(k,n)$-diagrams of rank (number of through edges) strictly smaller
than $k$.  Since multiplication of diagrams cannot increase the number
of through edges, $\B^0_{k,n}$ is a sub-bimodule of $\B_{k,n}$, and
hence the quotient
\[
V^n_k = \B_{k,n}/ \B^0_{k,n}
\]
is a $(\B_k, \B_n)$-bimodule. The set of $(k,n)$-diagrams of rank $k$
is a complete set of representatives of the quotient. If $k=n$, then
$V^n_n \cong \Bbbk \Sym_n$ and $\F_n(V^n_n) = V^n_n \xi_n =
0$. Furthermore, if $k<n$, then we have an isomorphism
\begin{equation}\label{eq:F_nV}
  \F_n(V^n_k) = V^n_k\xi_n \cong V^{n-2}_k
\end{equation}
as $(\B_k,\B_{n-2})$-bimodules. The isomorphism arises from forgetting
the rightmost horizontal edge in $b\xi_n$, for each $(k,n)$-diagram
$b$. (There is a factor of $\parm^{-1}$ which does not matter.)  By
restriction, since $\Bbbk\Sym_k$ is contained in $\B_k$, the bimodule
$V^n_k$ is a $(\Bbbk\Sym_k, \B_n)$-bimodule. Therefore, if $\lambda
\vdash k$, we define
\[
M^{(\lambda,n)} = S^\lambda \otimes_{\Bbbk \Sym_k} V^n_k .
\]
This is a right $\B_n$-module, where $S^\lambda$ is the Specht module
considered in the previous section. If $\lambda \vdash n$, then $k=n$
and $V^n_n \cong \Bbbk \Sym_n$, so $M^{(\lambda, n)} \cong S^\lambda$
as right $\B_n$-modules (with $\B_n \xi_n \B_n$ acting trivially).
Clearly this is an irreducible $\B_n$-module; indeed, it is
irreducible as a $\Bbbk\Sym_n$-module.

\begin{prop}\label{prop:irr-B_n}
  A full set of irreducible right $\B_n$-modules is the set of
  $M^{(\lambda,n)}$ such that $\lambda \vdash k$, $0 \le k \le n$, and
  $n$, $k$ are of the same parity.
\end{prop}

\begin{proof}
Assume that $k,n$ have the same parity.  To show that the
$M^{(\lambda,n)}$ are pairwise non-isomorphic and irreducible, we 
proceed by induction. We consider the two cases $k<n$ and $k=n$. 
(Modules between the two cases are non-isomorphic by Theorem 
\ref{thm:Green}.)

If $k<n$ and $\lambda \vdash k$, then it follows from
\eqref{eq:rec-iso}, \eqref{eq:F_nV}, and the definition of
$M^{(\lambda,n)}$ that
\[
\F_n(M^{(\lambda,n)}) = M^{(\lambda,n)} \xi_n = S^\lambda \otimes_{\Bbbk
  \Sym_k} V^n_k \xi_n \cong S^\lambda \otimes_{\Bbbk
  \Sym_k} V^{n-2}_k = M^{(\lambda,n-2)}
\]
as right $\B_{n-2}$-modules. Since $\B_n$ is semisimple, and
$M^{(\lambda,n-2)} \ne 0$ by the inductive hypothesis, it follows that
\[
\G_{n-2} (M^{(\lambda,n-2)}) \cong M^{(\lambda,n)}
\]
as right $\B_n$-modules. Furthermore, by Proposition
\ref{prop:AeA=eAe}(a), $M^{(\lambda,n)}$ is irreducible as a right
$\B_{n}$-module. Appealing to Proposition \ref{prop:AeA=eAe}(b), we
see that the distinct $M^{(\lambda,n)}$ are pairwise non-isomorphic.

In the case $k=n$ and $\lambda \vdash n$, we have
$M^{(\lambda,n)} \cong S^\lambda$. Such modules are pairwise 
non-isomorphic (and irreducible) by the remarks preceding the 
theorem. This completes the proof.
\end{proof}

\begin{rmk}
Although not needed in the sequel, to complete the picture we describe
a $\Bbbk$-basis for $M^{(\lambda,n)}$.  This requires finding a
complete set of orbit representatives for the left action of $\Sym_k$
on the set of $(k,n)$-diagrams of rank $k$.  If $b$ is a
$(k,n)$-diagram, we let $\pi(b)$ in $\Sym_k$ be the permutation
obtained from $b$ by removing the horizontal edges and their
endpoints.  Recall from \cites{Fishel-Grojnowski,Xi} that a
$(k,n)$-diagram $b$ is a \emph{flat $(k,n)$-dangle} if $\pi(b)$ is the
identity.  Then the set of flat $(k,n)$-dangles is the desired set of
representatives.  Any $(k,n)$-diagram $b$ is uniquely expressible as a
product
\[
b = \pi(b) d(b),
\]
where $d(b)$ is a flat $(k,n)$-dangle $d(b)$. It follows that the set
\begin{equation*}
\{ v \otimes d \mid v \in \widehat{S}^\lambda,\ d \text { a flat
  $(k,n)$-dangle} \}
\end{equation*}
is a $\Bbbk$-basis for $M^{(\lambda,n)}$, where $\widehat{S}^\lambda$ is any
$\Bbbk$-basis of $S^\lambda$.
\end{rmk}

Next, we explain why the family $\{ \B_n \mid n\ge 0 \}$ is
multiplicity-free. Recall that if $B$ is a subalgebra of an algebra
$A$ and if $M$ is a right $B$-module, then the induced module is the
right $A$-module $\Ind_B^A M$ defined by $\Ind_B^A M = M \otimes_B
A$. The functor $\Ind_B^A$ from $B$-modules to $A$-modules is a left
adjoint to the usual restriction functor $\Res^A_B$ from $A$-modules
to $B$-modules, meaning that Frobenius reciprocity holds:
\[
\Hom_A(\Ind_B^A M, N) \cong \Hom_B(M, \Res^A_B N),
\]
where $M$ is any right $B$-module, $N$ any right $A$-module.

We can apply these generalities to the inclusion $\iota: \B_{n-1}
\into \B_n$, which identifies $\B_{n-1}$ with a subalgebra of $\B_n$.
Wenzl \cite{Wenzl} observed that $\xi_n \B_n = \xi_n \B_{n-1}$ and
also that the map
\begin{equation}\label{eq:xi_n-iso}
  \B_{n-1} \to \xi_n \B_n \text{ defined by } x \mapsto \xi_n x
\end{equation}
gives an isomorphism $\B_{n-1} \cong \xi_n \B_n$ of
$(\B_{n-2},\B_{n-1})$-bimodules. Note that $\xi_n \B_{n-1}$ is a left
$\B_{n-2}$-module since $\xi_n$ commutes with $\B_{n-2}$, by
\eqref{eq:xi_n-commutes}. Let $\lambda \vdash k$ where $k<n$ and $k$
has the same parity as $n$. If we restrict the $\B_n$-module
isomorphism
\begin{equation*}
  M^{(\lambda,n)} \cong \G_{n-2}(M^{(\lambda,n-2)}) =
  M^{(\lambda,n-2)} \otimes_{\B_{n-2}} \xi_n \B_n
\end{equation*}
to $\B_{n-1}$, it follows that
\begin{equation}
  \Res^{\B_n}_{\B_{n-1}} M^{(\lambda,n)} \cong \Ind_{\B_{n-2}}^{\B_{n-1}} M^{(\lambda,n-2)}
\end{equation}
as right $\B_{n-1}$-modules. In light of Frobenius reciprocity this
says that
\[
\Hom_{\B_{n-1}}(\Res M^{(\lambda,n)}, M^{(\mu,n-1)}) \cong
\Hom_{\B_{n-2}}(M^{(\lambda,n-2)}, \Res M^{(\mu,n-1)})
\]
for any $\mu \vdash l \le n-1$ where $l$ has the same parity as
$n-1$. Here, we omitted the sub and superscripts on the restriction
functors for readability. Since the algebras are semisimple,
this says that
\begin{equation}\label{eq:res-reciprocity}
[\Res M^{(\lambda,n)}: M^{(\mu,n-1)}]_{n-1} = [\Res M^{(\mu,n-1)}:
  M^{(\lambda,n-2)}]_{n-2\,},
\end{equation}
where we write $[M: S]_n$ for the multiplicity of an irreducible
$\B_n$-module $S$ in another $\B_n$-module $M$. By induction, we may
assume that the right hand side of \eqref{eq:res-reciprocity} is
always 0 or 1. This shows that restriction from $\B_n$ to $\B_{n-1}$
is multiplicity-free, at least for the case of $k<n$.

If $k=n$ and $\lambda \vdash n$, then $M^{(\lambda,n)} \cong S^\lambda$
with $\B_n\xi_n \B_n$ acting trivially. That is, its restriction to
$\B_{n-1}$ is a module with $\B_{n-1} \xi_{n-1} B_{n-1} \subset
\B_n\xi_n \B_n$ acting trivially, so the restriction is a $\Bbbk
\Sym_{n-1}$-module. This means that the restriction rule in this case
is the same as the usual restriction rule for symmetric groups (which
is also multiplicity-free). This completes the proof that the family
$\{\B_n \mid n\ge 0\}$ is a multiplicity-free family, in
the sense of Definition \ref{def:mfa}. 

In fact, the above analysis shows that the restriction of an
irreducible $\B_n$-module $M^{(\lambda,n)}$ to $\B_{n-1}$ breaks up
into a direct sum of irreducible $\B_{n-1}$-modules $M^{(\mu,n-1)}$
indexed by all partitions $\mu$ obtained from $\lambda$ by removing or
adding one box. This justifies the branching graph for this family,
which is displayed in Figure \ref{Bratteli-Brauer} below.

\begin{figure}[htb]
\begin{center}
\begin{tikzpicture}[xscale=2.75*\UNIT, yscale=-4.25*\UNIT]
	\coordinate (0) at (0,.5);
	\coordinate (11) at (.8,1.4);
	\foreach \x in {1, ..., 3}{\coordinate (2\x) at (-2+2*\x,2.2);}
	\foreach \x in {1, ..., 5}{\coordinate (3\x) at (-1.6+2.6*\x,3.24);}
	\foreach \x in {1, ..., 8}{\coordinate (4\x) at (-1.9+1.9*\x,4.8);}
	\foreach \x in {1, ..., 11}{\coordinate (5\x) at (-.66+1.66*\x,7.0);}
\draw (0)--(11);
\draw (11)--(21) (11)--(22) (11)--(23);
\draw (21)--(31);
\draw (22)--(31) (22)--(32) (22)--(33);
\draw (23)--(31) (23)--(33) (23)--(34);
\draw (31)--(41) (31)--(42) (31)--(43);
\draw (32)--(42) (32)--(44) (32)--(45);
\draw (33)--(42) (33)--(43) (33)--(45) (33)--(46) (33)--(47);
\draw (34)--(43) (34)--(47) (34)--(48);
\draw (41)--(51);
\draw (42)--(51) (42)--(52) (42)--(53);
\draw (43)--(51) (43)--(53) (43)--(54);
\draw (44)--(52) (44)--(55) (44)--(56);
\draw (45)--(52) (45)--(53) (45)--(56) (45)--(57) (45)--(58);
\draw (46)--(53) (46)--(57) (46)--(59);
\draw (47)--(53) (47)--(54) (47)--(58) (47)--(59) (47)--(510);
\draw (48)--(54) (48)--(510) (48)--(511);
\begin{scope}[every node/.style={fill=white}]
	\node at (0) {\footnotesize$\bb$};
	\node at (11) {\tPART{1}};

	\node at (23) {\tPART{1,1}};
	\node at (22) {\tPART{2}}; 
	\node at (21) {\footnotesize$\bb$};

	\node at (34) {\tPART{1,1,1}};
	\node at (33) {\tPART{2,1}};
	\node at (32) {\tPART{3}};
	\node at (31) {\tPART{1}};

	\node at (48) {\tPART{1,1,1,1}};
	\node at (47) {\tPART{2,1,1}};
	\node at (46) {\tPART{2,2}};
	\node at (45) {\tPART{3,1}};
	\node at (44) {\tPART{4}};
	\node at (43) {\tPART{1,1}};
	\node at (42) {\tPART{2}};
	\node at (41) {\footnotesize$\bb$};

	\node at (511) {\tPART{1,1,1,1,1}};
	\node at (510) {\tPART{2,1,1,1}};
	\node at (59) {\tPART{2,2,1}};
	\node at (58) {\tPART{3,1,1}};
	\node at (57) {\tPART{3,2}};
	\node at (56) {\tPART{4,1}};
	\node at (55) {\tPART{5}};
	\node at (54) {\tPART{1,1,1}};
	\node at (53) {\tPART{2,1}};
	\node at (52) {\tPART{3}};
	\node at (51) {\tPART{1}};
\end{scope}
\node at (1,7.5) {$\vdots$};
\foreach \x in {8,15} {\node at (\x , 7.6) {$\ddots$}; }
\end{tikzpicture}
\end{center}
\caption{Branching graph for the family $\{ \B_n \}$}
\label{Bratteli-Brauer}
\end{figure}
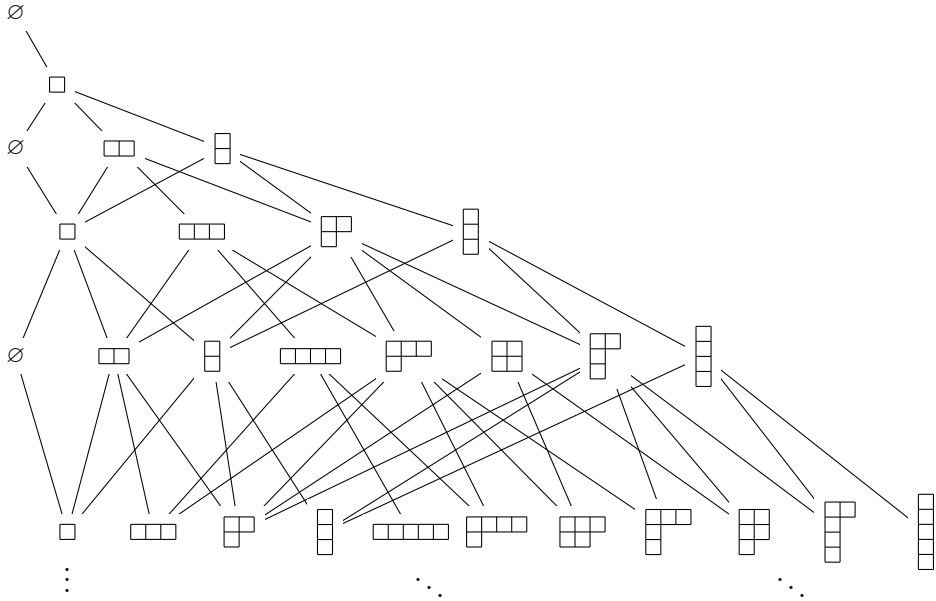

Since $\{ \B_n \mid n \geq0 \}$ is a multiplicity-free family, each
$\B_n$ has canonical idempotents $\idem_\T$ given by Definition
\ref{def:idem-T}. These idempotents are indexed by paths $\T$ of
length $n$ in the branching graph. The set $\Tab(n)$ may be
identified with the set of \demph{up-down tableaux}, which are
sequences of partitions of the form
\begin{equation}\label{eq:up-down-tableau}
  \T = (\lambda_0, \lambda_1, \cdots, \lambda_{n-1}, \lambda_{n})
\end{equation}
such that $\lambda_0 = \bb$ and, for each $k$, the partition
$\lambda_{k+1}$ is obtainable from the preceding partition
$\lambda_{k}$ by adding or removing exactly one box.

We wish to compute the idempotents $\idem_\T$ by means of a sequence
of JM-elements, according to the results of Section \ref{sec:jm}.
Following Nazarov \cite{Nazarov}, we define elements $J_k \in \B_k$
($k\ge1$) by
\begin{equation}
  J_k = \sum_{i = 1}^{k-1} (i,k) - \sum_{i = 1}^{k-1}
  \overline{(i,k)}.
\end{equation}
We define $J_1$ as zero. Our definition of these elements differs
slightly from Nazarov's, in that we have removed an unnecessary shift
by $(\parm-1)/2$.  For any $n \in \N$, the elements $J_1, \dots, J_n$
may be regarded as elements of $\B_n$ by means of the embeddings $\B_1
\subset \cdots \subset \B_{n-1} \subset \B_n$. The following easy
results can be checked by direct computations.

\begin{lem}[\cite{Nazarov}*{Lemma~2.1}]\label{lem:Naz2.1}
  For any $k = 1, \dots, n$, the element $J_k$ commutes with any $b \in
  \B_{n-1}$. Hence, the elements $J_1, \dots, J_n$ pairwise commute in
  $\B_n$.
\end{lem}

We omit the easy proof, which is given in \cite{Nazarov}. The lemma
immediately gives the following commutation relations between the
$J_k$ and the generators $s_i$, $e_i$ defined in
\eqref{eq:s_i-e_i}. Note that the relations in part (c) differ from
those given by Nazarov because our definition of $J_k$ differs
slightly from his.

\begin{prop}[\cite{Nazarov}*{Prop.~2.3}]\label{prop:Naz2.3}
  The following relations hold in the algebra $\B_n$:
  \begin{enumerate}
   \item\  $s_k J_l = J_l s_k$,\quad $e_k J_l = J_l e_k$ \qquad $(l \ne
     k, k+1)$.  
   \item\  $s_k J_k - J_{k+1} s_k = e_k - 1$, \quad $s_k J_{k+1} - J_k
     s_k = 1 - e_k$.
   \item\  $e_k(J_k+J_{k+1}) = (1-\parm)e_k = (J_k+J_{k+1}) e_k$.
  \end{enumerate}
\end{prop}

\begin{proof}
The commutation relations (a) follow from Lemma \ref{lem:Naz2.1} if
$l>k+1$ and from the definitions otherwise. Furthermore, it is easy to
check from the definition that the elements $J_k$ can be defined by
the recursion
\[
  J_1 = 0, \qquad J_{k+1} = s_kJ_ks_k + s_k - e_k \ \ (k \ge 1). 
\]
This implies the relations (b). Turning to (c), we have by direct
computation for any $l = 1, \dots, k-1$ the equalities
\[
  e_k\, (k,l) = e_k\, \ov{(k+1,l)} \quad\text{and}\quad e_k\,
  \ov{(k,l)} = e_k\, (k+1,l).
\]
Combining these equalities with the obvious identities $e_k s_k =
e_k$, $e_k^2 = \parm e_k$ and the definition of the $J_k$ produces the
leftmost equality in (c). The rightmost equality in proved similarly.
\end{proof}

Relations (a) and (b) of the proposition immediately imply the following.

\begin{cor}\cite{Nazarov}*{Cor.~2.4}\label{cor:Naz2.4}
  The sum $z_n = J_1 + \cdots + J_{n-1} + J_n$ is a central element of
  $\B_n$.
\end{cor}

It remains to compute the eigenvalues of the $J_k$ on the irreducible
modules and prove that the sequence $(J_k)_{k \in \N}$ is separating.

\begin{prop}\label{prop:z_n-action}
  Let $\lambda \vdash k$ where $k=n-2l$ and $0 \le 2l \le n$.  Suppose
  that $a_\lambda$ is the eigenvalue of equation
  \eqref{eq:a_lambda-sum}. Then the central element $z_n = J_1 +
  \cdots + J_n$ of $\B_n$ acts on $M^{(\lambda,n)} = S^\lambda
  \otimes_{\Bbbk \Sym_k} V^n_k$ as the scalar $\beta_\lambda =
  a_\lambda + l(1-\parm)$.
\end{prop}

\begin{proof}
This argument follows the proof of
\cite{Goodman-Graber:JM}*{Theorems~5.3, 5.1}. We proceed by induction
on $n$. The base cases $n=0$, $1$ are trivial, so assume that $n \ge
2$. There are two cases.

If $l=0$, then $\lambda \vdash n$ and $M^{(\lambda,n)} \cong
S^\lambda$, with the ideal $\B_n \xi_n \B_n$ acting trivially. We can
write
\[
  z_n = z^{\Sym_n}_n - \overline{z}_n,
\]
where $z^{\Sym_n}_n = \sum_{i<j} (i,j)$ is the sum of all the
transpositions in $\Sym_n$ and $\ov{z}_n = \sum_{i<j} \ov{(i,j)} \in \B_n
\xi_n \B_n$. It follows from Proposition \ref{prop:z_n-Sym} that $z_n$
acts as the scalar $a_\lambda$, so the proof is complete in case $l=0$. 

Now suppose $l > 0$. In this case we use the isomorphism 
\[
  M^{(\lambda,n)} \cong \G_{n-2}(M^{(\lambda,n-2)}) = S^\lambda
  \otimes_{\Bbbk\Sym_k} V^{n-2}_k \otimes_{\xi_n \B_n \xi_n} \xi_n
  \B_n
\] 
from the proof of Proposition \ref{prop:irr-B_n}. Since the central
element $z_n$ acts by a fixed scalar on the entire module, it suffices
to compute its eigenvalue on any nonzero vector in the module, so we
consider its action on $u \otimes v \otimes \xi_n$, where $0 \ne u
\otimes v \in S^\lambda \otimes_{\Bbbk \Sym_k} V^{n-2}_k$. By induction we have
\[
  (u \otimes v) z_{n-2} = \big( a_\lambda + (l-1)(1-\parm) \big) \, u
\otimes v.
\]
It follows that
\begin{align*}
  (u \otimes v \otimes \xi_n) z_n &= (u \otimes v \otimes
  \xi_n)(z_{n-2} + J_{n-1} + J_n) \\ &= (u \otimes v \otimes
  \xi_n)z_{n-2} + (u \otimes v \otimes \xi_n)(J_{n-1} + J_n).
\end{align*}
By Proposition \ref{prop:Naz2.3}(a) we know that $\xi_n =
\frac{1}{\parm} e_{n-1}$ commutes with $z_{n-2}$, so the first term in
the right hand side of the above is
\[
  (u \otimes v \otimes \xi_n)z_{n-2} = (u \otimes v)z_{n-2} \otimes
\xi_n = \big( a_\lambda + (l-1)(1-\parm) \big) \, u\otimes v \otimes \xi_n.
\]
The second term in the right hand side is computed by Proposition
\ref{prop:Naz2.3}(c) as
\[
  (u \otimes v \otimes \xi_n)(J_{n-1} + J_n) = (1-\parm)\, u \otimes v
\otimes \xi_n.
\]
Hence, by combining the equations in the last three displays, we
obtain the equality
\[
  (u \otimes v \otimes \xi_n) z_n = \big( a_\lambda + l(1-\parm) \big)
\, u \otimes v \otimes \xi_n
\]
and the proof is complete.
\end{proof}

This result will now be applied to compute the eigenvalues of the
$J_k$ on the Gelfand--Tsetlin basis of the irreducible $\B_n$-modules.

\begin{prop}\label{prop:c_T}
  Suppose that $\lambda \in \Irr(n)$ and $\{v_\T \mid \T \mapsto
  \lambda\}$ is the Gelfand--Tsetlin basis of $M^{(\lambda,n)}$. Let
  $\T = (\lambda_0,\lambda_1,\ldots, \lambda_n)$ be an up-down tableau
  with $\lambda_n = \lambda$. Suppose that $\lambda_k$ and
  $\lambda_{k-1}$ differ by a box in row $i$ and column $j$. Then the
  eigenvalue of $J_k$ on the eigenvector $v_\T$ is
  \[
    c_\T(k) = 
    \begin{cases}
       j-i & \text{ if $\lambda_k$ has one more box than
         $\lambda_{k-1}$}, \\ 
       (1-\parm) + i-j & \text{ if $\lambda_k$ has one
         fewer box than $\lambda_{k-1}$}.
    \end{cases}
  \]
\end{prop}

\begin{proof}
Set $z_k = \sum_{l=1}^k J_l$ and note that $J_k = z_k - z_{k-1}$, for
any $1\leq k\leq n$. 

We proceed by induction on $n$. For $n=1$ the result is clear:
$c_\T(1) = 0$ as $J_1 = 0$. Let $n > 1$ and let $\T \in \Tab(n)$.  By
the inductive hypothesis, $c_{\ov{\T}}(k)$ has the desired value for
any $k \le n-1$. By Proposition \ref{prop:GJM-separation}(a), $c_\T(k)
= c_{\ov{\T}}(k)$ for all $k<n$, so $c_\T(k)$ has the desired value
for all $k<n$. Thus, it suffices to compute the value $c_\T(n)$.

By Propositions \ref{prop:central-J_n-property} and
\ref{prop:z_n-action} we have $c_\T(n) = \beta_\lambda - \beta_\mu$,
where $\ov{\T} \mapsto \mu$, and $\beta_\lambda = a_\lambda +
l(1-\parm)$. There are two cases to consider: if $\lambda = \lambda_n$
has one more box or one fewer box than $\mu = \lambda_{n-1}$. In the
first case, Propositions \ref{prop:z_n-action} and \ref{prop:c_T-symm}
give us $\beta_\mu = a_\mu + l(1-\parm)$, and
\[
  c_\T(n) = \beta_\lambda - \beta_\mu = j-i. 
\]
In the second case, $\beta_\mu = a_\mu + (l-1)(1-\parm)$, and hence
\[
  c_\T(n) = \beta_\lambda - \beta_\mu = (1-\parm) + i-j.
\]
This complete the proof.
\end{proof}

\begin{cor}\label{cor:Brauer-JM-sep}
  The sequence $(J_k \mid k \in \N)$ is a JM-sequence in the sense of
  Definition \ref{def:JM-defn}.
\end{cor}

\begin{proof} 
Since $J_k = z_k - z_{k-1}$ and $z_k \in Z(\B_k)$ for all 
$1\leq k \leq n$ (Corollary \ref{cor:Naz2.4}), it follows that each 
$J_k\in\X_n$ and that $(J_k)_{k\in\N}$ is additively central. To prove 
that it is also a separating sequence, we use Proposition 
\ref{prop:separation-generation-equivalent}. That is, we verify that 
$\TS = \T$ if and only if $c_\TS = c_\T$. (One direction is automatic.)

Proposition \ref{prop:c_T} computes the content vectors $c_\T =
(c_\T(1), \dots, c_\T(n))$ for each $\T \in \Tab(n)$. If $\T =
(\lambda_0 \to \cdots \to \lambda_n)$, we write $\T[k] = (\lambda_0 \to
\cdots \to \lambda_k)$ for the truncated path. Assume $\TS\neq \T$ are
distinct paths of length $n$ and find the first level $k \leq n$ at
which the paths $\TS, \T$ diverge. So $\TS[k-1] = \T[k-1]$, yet
$\TS[k] \neq \T[k]$.  Let $\TS[k] \mapsto \lambda$ and $\T[k] \mapsto
\nu$ be the terminal shapes of the paths, and let $\T[k-1] \mapsto \mu$.  There are three cases.

Case 1: $\lambda$, $\nu$ are obtained by adding different boxes to $\mu$.
Here $c_\TS(k)$ and $c_\T(k)$ are both computed using the first formula in 
Proposition \ref{prop:c_T}. Appealing to Remark \ref{rmk:addable}, 
we see that $c_\TS(k) \neq c_\T(k)$.

Case 2: $\lambda$, $\nu$ are obtained by removing different boxes from
$\mu$. We must use the second formula in Proposition \ref{prop:c_T}. 
Appealing to Remark \ref{rmk:addable}, we again have $c_\TS(k) \neq c_\T(k)$.

Case 3: One of $\lambda$, $\nu$ is obtained by adding a box and the
other by removing one.  Here $c_\TS(k)$ and $c_\T(k)$ cannot possibly
be equal, as Proposition \ref{prop:c_T} says that one value is an
integer and the other is not (recall that $\delta\in\Bbbk \setminus \Z$).

All cases reach the conclusion that $c_\TS \neq c_\T$, so the
proof is complete.
\end{proof}


\begin{examples}
To avoid ambiguity, we write $\idem^{(n)}(\lambda)$ for the primitive
central idempotent $\idem(\lambda)$ in $\B_n$.  The
$\idem^{(n)}(\lambda)$ can be computed recursively using Theorem
\ref{thm:central-P-recursion} and Proposition \ref{prop:c_T},
referring to the branching graph in Figure \ref{Bratteli-Brauer}. Of course
$\idem^{(1)}(\ttPART{1}) = 1$. 


\smallskip \noindent\emph{Primitive central idempotents for $n=2$:}
\begin{align*}
\idem^{(2)\!}({\bb}) &= P^{\,\bb}_{\ttPART{1}} \,\idem^{(1)}(\ttPART{1}) = P^{\,\bb}_{\ttPART{1}} = {\tfrac{J_2 - 1}{(1-\parm) -1}\cdot \tfrac{J_2 + 1}{(1-\parm) + 1}} \\[.5ex]
\idem^{(2)\!}(\ttPART{2}) &= P^{\,\ttPART{2}}_{\ttPART{1}}\,\idem^{(1)}(\ttPART{1}) = P^{\,\ttPART{2}}_{\ttPART{1}} = {\tfrac{J_2 - (1-\parm)}{1 - (1-\parm)} \cdot \tfrac{J_2 + 1}{1+1}} \\[.5ex]
\idem^{(2)\!}(\ttPART{1,1}) &= P^{\,\ttPART{1,1}}_{\ttPART{1}} \,\idem^{(1)}(\ttPART{1}) = P^{\,\ttPART{1,1}}_{\ttPART{1}}  = {\tfrac{J_2 - (1-\parm)}{-1-(1-\parm)} \cdot \tfrac{J_2 - 1}{-1-1}}. 
\\[1ex]
\intertext{\emph{Primitive central idempotents for $n=3$:}}
\idem^{(3)\!}(\ttPART{1}) &= P^{\,\ttPART{1}}_{\ttPART{2}} \,\idem^{(2)\!}(\ttPART{2}) + P^{\,\ttPART{1}}_{\ttPART{1,1}} \,\idem^{(2)\!}(\ttPART{1,1}) + P^{\,\ttPART{1}}_{\bb} \,\idem^{(2)\!}(\bb) \\[.25ex]
&= {\tfrac{(J_3+1)(J_3-2)}{(\parm+2)(\parm-1)}} \idem^{(2)\!}(\ttPART{2}) 
	+ {\tfrac{(J_3+2)(J_3-1)}{(\parm-1)(\parm-4)}} \idem^{(2)\!}(\ttPART{1,1}) 
	+ 1\cdot\idem^{(2)\!}(\bb)\\[1ex]
\idem^{(3)\!}(\ttPART{3}) &= P^{\,\ttPART{3}}_{\ttPART{2}} \,\idem^{(2)\!}(\ttPART{2}) 
	= {\tfrac{(J_3+\parm)(J_3+1)}{3(\parm+2)}} \idem^{(2)\!}(\ttPART{2})  \\[1ex]
\idem^{(3)\!}(\ttPART{2,1}) &= P^{\,\ttPART{2,1}}_{\ttPART{2}} \,\idem^{(2)\!}(\ttPART{2}) + P^{\,\ttPART{2,1}}_{\ttPART{1,1}} \,\idem^{(2)\!}(\ttPART{1,1}) \\[.25ex]
&= -{\tfrac{(J_3+\parm)(J_3-2)}{3(\parm-1)}} \idem^{(2)\!}(\ttPART{2}) + {\tfrac{(J_3+\parm-2)(J_3+2)}{3(\parm-1)}} \idem^{(2)\!}(\ttPART{1,1}) \\[1ex]
\idem^{(3)\!}({\!}_{\!}\raisebox{-.5ex}{\ttPART{1,1,1}}{\!}_{\!}) &= P^{\,\ttPART{1,1,1}}_{\ttPART{1,1}} \,\idem^{(2)\!}(\ttPART{1,1}) 
	=  - {\tfrac{(J_3+\parm-2)(J_3-1)}{3(\parm-4)}} \idem^{(2)\!}(\ttPART{1,1}) .
\\[1ex]
\intertext{\emph{Primitive central idempotents for $n=4$:}\newline
There are eight idempotents at this level; we compute two of them:}
\idem^{(4)\!}(\ttPART{2}) &= 
	P^{\,\ttPART{2}}_{\ttPART{1}} \,\idem^{(3)\!}(\ttPART{1}) 
	+ P^{\,\ttPART{2}}_{\ttPART{3}} \,\idem^{(3)\!}(\ttPART{3}) 
	+ P^{\,\ttPART{2}}_{\ttPART{2,1}} \,\idem^{(3)\!}(\ttPART{2,1}) \\[.25ex]
&= {\tfrac{(J_4+\parm-1)(J_4+1)}{2\parm}}\idem^{(3)\!}(\ttPART{1}) 
	+ {\tfrac{(J_4+1)(J_4-3)}{(\parm+4)\parm}}\idem^{(3)\!}(\ttPART{3}) 
	- {\tfrac{(J_4+\parm)(J_4^2-4)J_4}{2(\parm-2)(\parm-4)\parm}}\idem^{(3)\!}(\ttPART{2,1})\\[1ex]
\idem^{(4)\!}(\ttPART{3,1}) &= P^{\,\ttPART{3,1}}_{\ttPART{2,1}} \,\idem^{(3)\!}(\ttPART{2,1}) 
	+ P^{\,\ttPART{3,1}}_{\ttPART{3}} \,\idem^{(3)\!}(\ttPART{3}) \\[.15ex]
	&= {\tfrac{(J_4+\parm)(J_4+\parm-2)(J_4+2)J_4}{8(\parm+2)\parm}} \,\idem^{(3)\!}(\ttPART{2,1})
	- {\tfrac{(J_4+\parm+1)(J_4-3)}{4\parm}} \,\idem^{(3)\!}(\ttPART{3}).
\end{align*}
We note that the summands in each $\idem^{(n)}(\lambda)$ are the
various $\idem_\T$ in that block, so the $\idem_\T$ are recoverable
from the above expressions.
\end{examples}

\begin{rmk}
  The recent preprint \cite{King-Martin-Parker} explores a completely
  different technique for computing central idempotents in semisimple
  Brauer algebras. Their technique is specific to that context.
\end{rmk}

\appendix
\renewcommand{\theequation}{\arabic{equation}} 
\section{Primitive central idempotents via trace characters}
\label{sec:pci}
\noindent
We give a brief exposition of another approach to computing the
primitive central idempotents in a split semisimple finite
dimensional algebra $\A$. The approach generalizes a classical formula 
of Frobenius for the central idempotents of group algebras $\C G$ for finite 
groups $G$ (see Corollary \ref{cor:Frobenius} below) in terms of the 
irreducible characters of $G$. 
We show that the irreducible trace characters of $\A$ still uniquely determine
its central idempotents, provided its defining field $\Bbbk$ has characteristic zero. 

Here, it is not necessary that $\A$ fits into a multiplicity-free family. The requirement on $\Bbbk$ guarantees invertibility of the $(\dim \A)\times(\dim \A)$ matrix of the natural trace form on $\A$. 
A slightly more general result (due to Kilmoyer) can be found in \cite{CR:Methods}*{Proposition (9.17)}; see also \cite{Ram:thesisChap1}.

\begin{defn*}Given any (not
necessarily irreducible) finite dimensional $\A$-module $V$, let $\chi^V$ be the \demph{trace
character} of $V$, defined by
\[
  \chi^V(a) = \trace(\varphi^V(a)),
\]
where $\varphi^V : \A \to \End_\Bbbk(V)$ is the representation
corresponding to the $\A$-module $V$. If $[V]=\lambda$ for
$\lambda\in\Irr(\A)$, we write $\chi^\lambda$ in place of
$\chi^{V}$. 
\end{defn*}
Let $\rho = \chi^{\A}$ be the trace character of the left 
regular module; i.e., the character of $\A$ regarded as a
module over itself by left multiplication.  Since
$\End_\Bbbk(V^\lambda) \cong (V^\lambda)^* \otimes V^\lambda$, it
follows from \eqref{eq:ss-dec} that
\begin{equation*} 
  \A \cong \textstyle \bigoplus_\lambda (\dim V^\lambda) V^\lambda 
\end{equation*}
as left $\A$-modules.  Since characters are additive on direct sums of
modules and since $\chi^\lambda(1) = \dim V^\lambda$, it follows that
\begin{equation}\label{eq:rho}
  \rho = \textstyle \sum_\lambda (\dim V^\lambda)\, \chi^\lambda =
  \sum_\lambda \chi^\lambda(1) \chi^\lambda. 
\end{equation}

The problem of finding central idempotents $\idem(\lambda)$ is now framed as follows. Given a fixed basis $B=B(\A)$ of $\A$, write 
\begin{equation}\label{eq:idem-via-basis}
  \idem(\lambda) = \textstyle \sum_{b \in B} c_b^\lambda \, b
\end{equation}
and try to compute the coefficients $c^\lambda_b \in \Bbbk$. To that end, we may multiply both sides of \eqref{eq:idem-via-basis} by
a basis element $b' \in B$, and then apply $\rho$ to both
sides to get
\begin{equation}\label{eq:idem-b'}
  \rho(\idem(\lambda) b') = \textstyle \sum_{b\in B} \rho(bb')\, c_b^\lambda .
\end{equation}
On the other hand, we can use \eqref{eq:rho} to express
$\rho(\idem(\lambda) b')$ as
\begin{equation}\label{eq:idem-b'-alt}
  \rho(\idem(\lambda) b') = \textstyle \sum_\mu \chi^\mu(1)
  \chi^\mu(\idem(\lambda) b') = \chi^\lambda(1) \,\chi^\lambda(b').
\end{equation}
(The last equality in \eqref{eq:idem-b'-alt} comes by multiplying the
equation $1 = \sum_\mu \idem(\mu)$ on the right by $b'$, then applying $\chi^\lambda$
to both sides.) Note that $\chi^{\lambda}(\idem(\mu)b') = 0$ for
$\lambda \ne \mu$, since
$\idem(\mu)b'$ belongs to a block upon which $\chi^\lambda$ acts as zero.

Finally, we combine \eqref{eq:idem-b'} and \eqref{eq:idem-b'-alt} to
obtain
\begin{equation} \label{eq:system}
  \textstyle \sum_{b\in B} \rho(bb')\, c_b^\lambda = \chi^\lambda(1)
  \,\chi^\lambda(b').
\end{equation}
For fixed $\lambda$, we may regard \eqref{eq:system} as a
linear system (one equation for each $b'$) that govern the
values $c^\lambda_b$. This leads to the following result.

\begin{prop}\label{prop:invertible}
Suppose a split semisimple finite dimensional algebra $\A$ has underlying field $\Bbbk$ of characteristic zero.
Then 
the primitive central idempotents $\idem(\lambda)$ of $\A$ are uniquely
determined by its irreducible characters.
\end{prop}

\begin{proof}
Given a basis $B$ of $\A$, let $M = \left(\rho(bb')\right)_{b'\!,\,b \in B}$ be the square matrix of
coefficients in the linear system \eqref{eq:system}, with rows indexed
by $b'$ and columns by $b$. This is just the matrix of the natural bilinear trace form, 
  i.e., $(a, a') = \rho(aa') \ \forall\,a, a' \in \A,$
with respect to the basis $B$. As $\A$ is split semisimple over a field of characteristic zero, a classical
argument, as in \cite{Vinberg}*{Theorem 11.54}, shows that the trace form is
nondegenerate. Hence $M$ is invertible.

Let $r^\lambda$ be the column vector
$\left(\chi^\lambda(1) \,\chi^\lambda(b')\right)_{b' \in B}$. Then the column vector $\left(c^\lambda_b\right)_{b\in B}$ defining $\idem(\lambda)$ in \eqref{eq:idem-via-basis} is uniquely determined and equal to $M^{-1} r^\lambda$. 
\end{proof}

We note that the vector $r^\lambda$ in the proof of Proposition \ref{prop:invertible}
is just the $\lambda$-row of the
character table of $\A$, scaled by $\chi^\lambda(1) = \dim_\Bbbk
V^\lambda$. So we have an alternative method of producing the irreducible characters, provided this table is known. See, e.g., \cite{Ram-chars} for the case of Brauer algebras.

In the case of group algebras, Proposition \ref{prop:invertible}
recovers the classical formula of Frobenius (see
\cite{Frobenius}*{III, pp.~ 244--274}). In that case, the matrix $M$
of the natural trace form is easy to invert.

\begin{cor}[Frobenius]  \label{cor:Frobenius}
  Suppose that $\A = \Bbbk G$ is a split semisimple group algebra over
  a field $\Bbbk$ of characteristic zero, where $G$ is a finite
  group. Then for any $\lambda \in \Irr(\Bbbk G)$,
  \[
     \idem(\lambda) = \frac{1}{|G|} \sum_{g \in G} \chi^\lambda(1)\,
     \chi^\lambda(g^{-1})\, g\,.
  \]

\end{cor}

\begin{proof}
This follows from the observation that $\rho(g)$ is zero for any $g
\ne 1_G$, while $\rho(1_G) = |G|$, where $1_G$ denotes the identity
element of $G$. Indeed, let $B(\A) = G$ be the basis of $\A$ given by the
  group elements. Then the matrix $M = \left(\rho(gg')\right)_{g',\,g \in G}$ in the proof of the proposition is
  $|G|$ times the permutation matrix $P = \left(\delta_{g^{-1},g'}\right)_{g',\,g
    \in G}$, so $M^{-1} = \frac{1}{|G|} \transpose{P}$. 
Then
\[
   \rho(gg') = |G| \, \delta_{g^{-1},\, g'}
\]
in terms of the usual Kronecker delta. The formula for
$\idem(\lambda)$ now follows by an easy calculation.
\end{proof}

\end{document}